\documentclass[10pt]{amsart}
\usepackage{amssymb}
\usepackage{amscd}

\def\today{\number\day\space\ifcase\month\or   January\or February\or
   March\or April\or May\or June\or   July\or August\or September\or
   October\or November\or December\fi\   \number\year}

\theoremstyle{definition}
\newtheorem{thm}{Theorem}[section]
\newtheorem{lem}[thm]{Lemma}
\newtheorem{prp}[thm]{Proposition}
\newtheorem{dfn}[thm]{Definition}
\newtheorem{cor}[thm]{Corollary}

\newtheorem{rmk}[thm]{Remark}

\newtheorem{exa}[thm]{Example}
\newtheorem{pbm}[thm]{Problem}

\newcommand{\beq}{\begin{equation}}
\newcommand{\eeq}{\end{equation}}
\newcommand{\beqr}{\begin{eqnarray*}}
\newcommand{\eeqr}{\end{eqnarray*}}
\newcommand{\bal}{\begin{align*}}
\newcommand{\eal}{\end{align*}}
\newcommand{\bei}{\begin{itemize}}
\newcommand{\eei}{\end{itemize}}

\newcommand{\af}{\alpha}
\newcommand{\bt}{\beta}
\newcommand{\gm}{\gamma}

\newcommand{\ep}{\varepsilon}
\newcommand{\zt}{\zeta}
\newcommand{\et}{\eta}

\newcommand{\te}{\theta}
\newcommand{\ld}{\lambda}

\newcommand{\kp}{\kappa}
\newcommand{\ph}{\varphi}

\newcommand{\ta}{\tau}

\newcommand{\Gm}{\Gamma}

\newcommand{\Th}{\Theta}

\newcommand{\Q}{{\mathbb{Q}}}
\newcommand{\Z}{{\mathbb{Z}}}
\newcommand{\R}{{\mathbb{R}}}
\newcommand{\C}{{\mathbb{C}}}
\newcommand{\N}{{\mathbb{Z}}_{> 0}}

\newcommand{\id}{{\mathrm{id}}}

\newcommand{\Prim}{{\mathrm{Prim}}}

\newcommand{\rank}{{\mathrm{rank}}}

\newcommand{\card}{{\mathrm{card}}}
\newcommand{\Aut}{{\mathrm{Aut}}}
\newcommand{\Ad}{{\mathrm{Ad}}}

\newcommand{\SL}{{\mathrm{SL}}}

\newcommand{\dirlim}{\varinjlim}

\newcommand{\Zq}[1]{\Z_{#1}}
\newcommand{\Zqt}{\Zq{2}}

\newcommand{\andeqn}{\,\,\,\,\,\, {\mbox{and}} \,\,\,\,\,\,}

\newcommand{\tfae}{the following are equivalent}
\newcommand{\ifo}{if and only if}

\newcommand{\ca}{C*-algebra}
\newcommand{\uca}{unital C*-algebra}

\newcommand{\hm}{homomorphism}

\newcommand{\am}{automorphism}
\newcommand{\fd}{finite dimensional}
\newcommand{\tst}{tracial state}
\newcommand{\hsa}{hereditary subalgebra}

\newcommand{\pj}{projection}
\newcommand{\mops}{mutually orthogonal \pj s}

\newcommand{\mvnt}{Murray-von Neumann equivalent}

\newcommand{\cp}{crossed product}
\newcommand{\tgca}{transformation group \ca}

\newcommand{\rp}{Rokhlin property}
\newcommand{\trp}{tracial Rokhlin property}

\newcounter{TmpEnumi}

\renewcommand{\S}{\subset}

\title[Freeness of actions on C*-algebras]{Freeness
    of actions of finite groups on C*-algebras}

\author{N.~Christopher Phillips}

\date{19~February 2009}

\address{Department of Mathematics, University  of Oregon,
       Eugene OR 97403-1222, USA.}

\email[]{ncp@darkwing.uoregon.edu}

\thanks{Research partially supported by NSF grant DMS-0701076.}

\begin{document}

\begin{abstract}
We describe some of the forms of freeness of group actions
on noncommutative C*-algebras that have been used,
with emphasis on actions of finite groups.
We give some indications of their strengths, weaknesses,
applications, and relationships to each other.
The properties discussed include the Rokhlin property,
K-theoretic freeness, the tracial Rokhlin property,
pointwise outerness, saturation,
hereditary saturation,
and the requirement that the strong Connes spectrum be the
entire dual.
\end{abstract}

\maketitle

\indent
Recall that an action $(g, x) \mapsto g x$
of a group $G$ on a space $X$
is free if whenever $g \in G \setminus \{ 1 \}$ and $x \in X,$
then $g x \neq x.$
That is, every nontrivial group element acts without fixed points.
So what is a free action on a \ca?

There are several reasons for being interested in free actions
on \ca s.
First,
there is the general principle of noncommutative topology:
one should find the \ca\  analogs of useful concepts from topology.
Free (or free and proper) actions of locally compact groups
on locally compact Hausdorff spaces have a number of good properties,
some of which are visible from topological considerations
and some of which become apparent only when one looks at
\cp\  \ca s.
We describe some of these in Section~\ref{Sec:Comm}.
Second,
analogs of freeness,
particularly pointwise outerness and the Rokhlin property,
have proved important in von Neumann algebras,
especially for the classification of group actions
on von Neumann algebras.
Again,
this fact suggests that one should see
to what extent the concepts and theorems
carry over to \ca s.
Third,
the classification of group actions on \ca s
is intrinsically interesting.
Experience both with the commutative case
and with von Neumann algebras
suggests that free actions are easier to understand and classify
than general actions.
(A free action of a finite group on a path connected space
corresponds to a finite covering space.)
Fourth,
noncommutative analogs of freeness
play an important role in questions
about the structure of crossed products.
Freeness hypotheses are important for results
on both simplicity and classifiability of crossed products.

It turns out that there are many versions of noncommutative freeness.
They vary enormously in strength,
from saturation (or full Arveson spectrum)
all the way up to the Rokhlin property.
The various conditions have different uses.
The main point of this article
is to describe some of the forms of noncommutative freeness
that have been used,
and give some indications of their strengths, weaknesses,
applications, and relationships to each other.
To keep things simple,
and to keep the focus on freeness,
we restrict whenever convenient to actions of finite groups.
For one thing, our knowledge is more complete in this case.
Also, we would
otherwise have to deal with noncommutative properness;
we discuss this issue briefly below.

We give a rough summary of the different conditions and their uses,
in approximate decreasing order of strength.
The strongest is free action on the primitive ideal space.
Outside the class of type~I \ca s,
this condition seems too strong for almost all purposes,
and we accordingly say little about it.
Next is the Rokhlin property.
This is the hypothesis in most theorems
on classification of group actions.
When the group is finite,
it also implies very strong structure preservation results
for crossed products.
K-theoretic freeness is close to the \rp,
at least when the K-theory is sufficiently nontrivial.
Unlike the \rp,
it agrees with freeness in the commutative case.
The \trp\  is weaker than the \rp,
and much more common;
its main use is in classification theorems for crossed products.
The main use of pointwise outerness, at least so far,
has been for proving simplicity of crossed products.
Hereditary saturation and having full strong Connes spectrum
are weaker conditions which give exactly what is needed
for crossed products by minimal actions to be simple.
Unfortunately, they are hard to verify.
Saturation is the condition which makes the crossed product
naturally Morita equivalent to the fixed point algebra.

For an action of a noncompact group,
there is a big difference between actions that are merely free
and those that are both free and proper.
Recall that an action of a locally compact group $G$
on a locally compact space $X$
is proper if for every compact set $K \subset X,$
the set
$\{ g \in G \colon g K \cap K = \varnothing \}$
is compact in~$G.$
Equivalently,
the map $(g, x) \mapsto (x, g x)$ is a proper map,
that is, inverse images of compact sets are compact.

One of the good things about a free action of a compact Lie group
on a locally compact space $X$
is that $X$ is a principal $G$-bundle over the orbit space $X / G.$
(See Theorem~\ref{CT:Bundle}.)
As a special case,
if $G$ is finite then $X$ is a covering space
(not necessarily connected) of $X / G.$
As discussed after Theorem~\ref{CT:Bundle},
this remains true for noncompact $G$ if the action is proper.
It fails otherwise.
One should compare the action of $\Z$ on $\R$ by translation
(free and proper)
with the action of $\Z$ on the circle $S^1$
generated by an irrational rotation
(free but not proper).
In the second case,
the orbit space is an uncountable set with the
indiscrete topology,
and the quotient Borel space is not countably separated.
However, free actions of this type are very important.
Here, for example, the crossed product is the well known irrational
rotation algebra.

There may be nearly as many versions of properness of actions on \ca s
as there are of freeness of actions on \ca s,
but the subject has been less well studied.
We refer to the work of Rieffel.
(For example, see~\cite{Rf3}.)
In this survey, we simply avoid the issue.
This is not to say that properness of actions on \ca s is not important.
Rather, it is a subject for a different paper.
In most situations in this paper in which the issue arises,
we will consider only the analog of freeness without properness.

Returning to freeness,
many of our examples will involve simple \ca s,
since much of what has been done has involved simple \ca s.
Indeed,
for one of our conditions,
the \trp,
a satisfactory definition is  so far known only in the simple case.
For similar reasons, we have much less to say
about actions on purely infinite \ca s
than about actions on stably finite \ca s.
It seems possible
(although proofs are still missing)
that the differences between some of our conditions
disappear in the purely infinite case.
See the discussion at the end of Section~\ref{Sec:Outer}.

This paper is organized as follows.
In Section~\ref{Sec:Comm},
we recall a number of theorems which characterize
freeness of actions of finite or compact groups
on compact or locally compact spaces.
Some involve \ca s,
while others are purely in terms of topology.
These results suggest properties
which might be expected of free actions on \ca s.
Several of them implicitly or explicitly motivate
various definitions of noncommutative freeness.
In each of the remaining four sections,
we discuss a notion of noncommutative freeness,
or a group of notions
which seem to us to be roughly comparable in strength
(with one exception: some of the conditions
in Section~\ref{Sec:Connes} are much weaker than the others).
See the further discussion at the end of Section~\ref{Sec:Comm}.
The order of sections is roughly from strongest to weakest.
The comparisons are inexact partly because
some versions of noncommutative freeness,
in their present form,
are useful only for restricted classes of groups or \ca s,
and sometimes there are no interesting examples in the overlap.
For example,
we don't know how to properly define the tracial Rokhlin property
for actions on \ca s which are not simple,
which makes it awkward to compare this property
with freeness of an action on a commutative \ca.

As will become clear in the discussion,
there are a number of directions in which further work is needed.

We describe some standard notation.
Throughout, groups will be at least locally compact.
All groups and spaces
(except primitive ideal spaces and spaces of irreducible
representations of \ca s)
will be Hausdorff.
For an action of a group $G$ on a locally compact space~$X,$
we let $C^* (G, X)$ denote the \tgca,
and we let $X / G$ denote the orbit space.
For an action $\af \colon G \to \Aut (A)$ of $G$ on a \ca~$A,$
written $g \mapsto \af_g,$
we denote the \cp\  by $C^* (G, A, \af).$
We further denote the fixed point algebra
\[
\{ a \in A \colon {\mbox{$\af_g (a) = a$ for all $g \in G$}} \}
\]
by $A^{\af},$ or by $A^G$ if $\af$ is understood.
The action $\af \colon G \to \Aut (C_0 (X))$
coming from an action of $G$ on~$X$
is $\af_g (f) (x) = f (g^{-1} x).$
Note that in this case $C_0 (X)^G$ can be canonically
identified with $C_0 (X / G).$
The restriction of $\af \colon G \to \Aut (A)$
to a subgroup~$H \subset G$
is $\af |_H,$
and the restriction to an invariant subalgebra~$B \subset A$
is $\af_{( \cdot )} |_B.$

All ideals in \ca s are assumed closed and two sided.
We will denote the cyclic group $\Z / m \Z$ by $\Zq{m}.$
(No confusion with the $m$-adic integers should occur.)
For a Hilbert space~$H,$
we denote by $L (H)$ and $K (H)$ the algebras of bounded and compact
operators on~$H.$

We would like to thank
Dawn Archey,
George Elliott,
Akitaka Kishimoto,
Hiroyuki Osaka,
Cornel Pasnicu,
Costel Peligrad,
Marc Rieffel,
Masamichi Takesaki,
and Dana Williams
for useful discussions and email correspondence.

\section{The commutative case}\label{Sec:Comm}

\indent
In the main part of this section,
we give some characterizations and properties of free actions
of finite (sometimes more general) groups on compact spaces.
In some parts of the rest of this survey,
we will concentrate on simple \ca s,
so, without proper interpretation,
what one sees here may provide little guidance.

\begin{thm}\label{CT:Bundle}
Let $G$ be a compact Lie group and let $X$ be a
locally compact $G$-space.
The action of $G$ on $X$ is free \ifo\  the map
$X \to X / G$ is the projection map of a principal $G$-bundle.
\end{thm}

A principal $G$-bundle is a locally trivial bundle with fiber $G,$
and where the transition maps between trivializations
are given by continuous maps to~$G,$
regarded as acting on itself by translation.

This result is also true for actions of locally compact
Lie groups which are free and proper.
See the theorem in Section~4.1 of~\cite{Ps}.
The definition of properness given there is different,
but for locally compact~$X$ it is equivalent.
See Condition~(5) in Theorem 1.2.9 of~\cite{Ps};
the notation is in the introduction of~\cite{Ps} and
Definition 1.1.1 there.

\begin{thm}\label{CT:Cantor}
Let $G$ be a finite group,
and let $X$ be a totally disconnected $G$-space.
The action of $G$ on $X$ is free \ifo\  $X$ is equivariantly
homeomorphic to a $G$-space of the form $G \times Y,$
where $G$ acts on itself by translation and acts trivially on $Y.$
\end{thm}

\begin{proof}
We first claim that for every $x \in X,$
there is a compact open set $K \subset X$ such that
$x \in K$ and the sets
$g K,$ for $g \in G,$ are disjoint.
To see this,
for each $g \in G$ choose disjoint compact open sets
$L_g$ and $M_g$ such that $x \in L_g$ and $g x \in M_g.$
Then take
\[
K = \bigcap_{g \in G \setminus \{ 1 \}} (L_g \cap g^{-1} M_g).
\]

Since $X$ is compact,
we can now find compact open sets $K_1, K_2, \ldots, K_n \subset X$
which cover $X$ and such that, for each~$m,$
the sets $g K_m,$ for $g \in G,$ are disjoint.
Set
\[
L_m = K_m \cap \left( X \setminus \bigcup_{g \in G}
        g (K_1 \cup K_2 \cup \cdots \cup K_{m - 1}) \right).
\]
(This set may be empty.)
One verifies by induction on~$m$
that the sets $g L_j,$ for $g \in G$ and $j = 1, 2, \ldots, m,$
are disjoint
and cover $\bigcup_{g \in G} g (K_1 \cup K_2 \cup \cdots \cup K_m).$
Set $Y = L_1 \cup L_2 \cup \cdots \cup L_n.$
Then the sets $g Y,$ for $g \in G,$ form a partition of~$X.$
The conclusion follows.
\end{proof}

\begin{thm}\label{CT:MorEq}
Let $G$ be a compact group and let $X$ be a locally compact $G$-space.
The action of $G$ on $X$ is free \ifo\  appropriate
formulas (see Situation~2 of~\cite{Rf2}) make a suitable completion of
$C_{\mathrm{c}} (X)$ into a $C_0 (X / G)$--$C^* (G, X)$
Morita equivalence bimodule.
\end{thm}

\begin{proof}
That freeness implies Morita equivalence is Situation~2 of~\cite{Rf2}.
(It actually covers proper actions
of locally compact but not necessarily compact groups.)
Both directions together follow from
Proposition~7.1.12 and Theorem~7.2.6 of~\cite{Ph1}.
\end{proof}

\begin{thm}\label{CT:Outer}
Let $G$ be a compact group and let $X$ be a
locally compact $G$-space.
The action of $G$ on $X$ is free \ifo\  for every
$g \in G \setminus \{ 1 \}$ and every
$g$-invariant ideal $I \subset C (X),$
the action of $g$ on $C (X) / I$ is nontrivial.
\end{thm}

This is really just a restatement of the requirement that $g$
have no fixed points.
It is included for comparison
with the conditions in Section~\ref{Sec:Outer}.

\begin{thm}\label{CT:Ideals}
Let $G$ be a compact group and let $X$ be a
locally compact $G$-space.
The action of $G$ on $X$ is free \ifo\  every ideal
$I \subset C^* (G, X)$ has the form $C^* (G, U)$
for some $G$-invariant open set $U \subset X.$
\end{thm}

\begin{proof}
Suppose the action of $G$ on $X$ is free.
The conclusion follows
(in fact, in the more general case of a free and proper
action of a locally compact group)
from Theorem~14 of~\cite{Gr1} and its proof.

Suppose the action of $G$ on $X$ is not free.
Our argument is very close to
the proof of Proposition~7.1.12 of~\cite{Ph1}.
Choose $x \in X$ and $g \in G \setminus \{ 1 \}$ such that $g x = x.$
Let $H \subset G$ be the subgroup given by
$H = \{ g \in G \colon g x = x \},$
and set $S = G x.$
Then $S$ is a closed subset of $X$ which is equivariantly
homeomorphic to $G / H.$
So, using Corollary~2.10 of~\cite{Gr2} for the
second isomorphism, $C^* (G, X)$ has a quotient
\[
C^* (G, S)
 \cong C^* (G, \, G / H)
 \cong K (L^2 (G / H)) \otimes C^* (H).
\]
Let $\pi \colon C^* (G, X) \to K (L^2 (G / H)) \otimes C^* (H)$
be the composition of this isomorphism with the quotient map.

Since $H$ is a nontrivial compact group,
$C^* (H)$ is not simple.
(For example, consider the kernel of the map to $\C$
induced by the one dimensional trivial representation.)
Let $J \subset C^* (H)$ be a nontrivial ideal.
Then $\pi^{-1} \big( K (L^2 (G / H)) \otimes J \big)$ is an ideal
in $C^* (G, X)$
which does not have the form $C^* (G, U)$
for any $G$-invariant open set $U \subset X.$
\end{proof}

For the next characterization,
recall the equivariant K-theory $K^*_G (X)$ of a locally compact
$G$-space $X,$ introduced in~\cite{Sg2}.
(Also see Section~2.1 of~\cite{Ph1}.)
It is a module over the representation ring $R (G)$
(see~\cite{Sg1}),
which can be thought of as the equivariant K-theory of a point,
or as the Grothendieck group of the abelian semigroup of
equivalence classes of finite dimensional unitary representations
of~$G,$
with addition given by direct sum.
The ring multiplication is tensor product.
There is a standard \hm\  $R (G) \to \Z$ which sends
a representation to its dimension,
and its kernel is called the augmentation ideal
and written $I (G).$
(See the example before Proposition~3.8 of~\cite{Sg1}.)
We also need localization of rings and modules,
as discussed in Chapter~3 of~\cite{AM}.
Our notation follows Part~(1) of the example on page~38 of~\cite{AM}.

The following two results are parts of Theorem~1.1.1 of~\cite{Ph1}.
They are essentially due to Atiyah and Segal
(Proposition~4.3 of~\cite{ASg} and Proposition~4.1 of~\cite{Sg2}).

\begin{thm}[Atiyah and Segal]\label{CT:KThy}
Let $G$ be a compact Lie group and let $X$ be a compact $G$-space.
The action of $G$ on $X$ is free \ifo\  for every prime ideal $P$
in the representation ring $R (G)$ which does not contain $I (G),$
the localization $K^*_G (X)_P$ is zero.
\end{thm}

\begin{thm}[Atiyah and Segal]\label{CT:KXG}
Let $G$ be a compact Lie group and let $X$ be a compact $G$-space.
The action of $G$ on $X$ is free \ifo\  %
the natural map $K^* (X / G) \to K^*_G (X)$ is an isomorphism.
\end{thm}

Freeness is also related to C*-index theory.
The C*-basic construction used in the following result
is found in Sections 2.1 and~2.2 of~\cite{Wt};
the two versions are the same by Lemma~2.2.9 of~\cite{Wt}.

\begin{thm}\label{T:IndFin}
Let $G$ be a finite group,
and let $X$ be a compact $G$-space.
Suppose the set of points $x \in X$ with trivial stabilizer $G_x$
is dense in $X.$
Define a conditional expectation $E \colon C (X) \to C (X)^G$
by
\[
E (f) (x) = \frac{1}{\card (G)} \sum_{g \in G} f (g x).
\]
Then $E$ has index-finite type in the sense of Watatani
(see Definition~1.2.2 of~\cite{Wt},
and see Lemma~2.1.6 of~\cite{Wt} for the \ca\  version)
\ifo\  the action of $G$ on $X$ is free.
Moreover, in this case, the C*-basic construction
gives an algebra isomorphic to $C^* (G, X).$
\end{thm}

\begin{proof}
See Propositions 2.8.1 and~2.8.2 of~\cite{Wt}.
\end{proof}

Some nontriviality condition on the action is necessary,
since one must rule out the trivial action
and the action of $G$ on its quotient $G / H$
by a subgroup~$H.$

Here is an example of preservation of structure
associated with freeness.

\begin{thm}\label{CT:Mf}
Let $G$ be a finite group.
If $G$ acts freely on a topological manifold~$M,$
then $M / G$ is a topological manifold.
If $M$ and the action are smooth,
then so is $M / G.$
\end{thm}

The theorem has more force in the smooth case.
For example, if $G$ is a finite cyclic group,
acting on $\R^2$ by rotation,
then $\R^2 / G$ is a topological manifold,
but is not smooth in a neighborhood of the orbit of~$0.$

The situation for actions on noncommutative \ca s
is much more complicated.
There are at least six rough categories of conditions which,
when properly defined,
at least approximately correspond
to freeness in the commutative case.
We list them in descending order of strength.
\begin{itemize}
\item
Free action on the primitive ideal space.
\item
The Rokhlin property
and the closely related property of K-freeness.
\item
The tracial Rokhlin property.
\item
Outerness.
\item
Hereditary saturation (full strong Connes spectrum),
and the closely related property of full Connes spectrum.
\item
Saturation.
\end{itemize}
At least outside the type~I case,
free action on the primitive ideal space
seems to be too strong a condition,
as we hope to persuade you in this survey.
One section below is devoted to each of the others,
except that we treat the last two together.
For each condition,
we give definitions,
say what it implies in the standard examples discussed below,
describe some applications,
describe how it is related to previously discussed conditions,
and say something about permanence properties.
We also state some open problems.

We will discuss a number of examples in this survey,
but we will use two kinds of examples systematically.
One kind is arbitrary actions on separable unital type~I \ca s.
For them, most of our freeness conditions turn out to
be equivalent to freeness of the induced action on the
primitive ideal space $\Prim (A)$ of the algebra $A.$
The other is product type actions on UHF~algebras.
We recall these, and give convenient conventions,
in the following definition.
For simplicity,
we stick to actions of $\Zqt.$

\begin{dfn}\label{D:PType}
For $n \in \N$ let $d (n)$
and $k (n)$ be integers
with $d (n) \geq 2$ and $0 \leq k (n) \leq \frac{1}{2} d (n).$
Choose \pj s $p_n, \, q_n \in M_{d (n)}$
such that $p_n + q_n = 1$ and $\rank (q_n) = k (n).$
The associated product type action
on the UHF~algebra
$A = \bigotimes_{n = 1}^{\infty} M_{d (n)}$
is the action
$\af \colon \Zqt \to \Aut (A)$
generated by the infinite tensor product automorphism
of order~$2$ given by
$\af = \bigotimes_{n = 1}^{\infty} \Ad ( p_n - q_n ).$
\end{dfn}

The conjugacy class of the action does not depend on the
choice of the \pj s $p_n$ and $q_n.$
We need only consider $0 \leq k (n) \leq \frac{1}{2} d (n),$
since replacing any particular $k (n)$ by $d (n) - k (n)$
gives a conjugate action.

To give an idea of what the various conditions mean,
let the notation be as in Definition~\ref{D:PType},
and consider the following specific cases:
\begin{itemize}
\item
If $d (n) = 2$ and $k (n) = 1$ for all~$n,$
then $\af$ has the Rokhlin property and is K-theoretically free.
See Examples \ref{E:PTypeRk} and~\ref{E:PTypeK}.
\item
If $d (n) = 3$ and $k (n) = 1$ for all~$n,$
then $\af$ has the tracial Rokhlin property,
but does not have the Rokhlin property and is not K-theoretically free.
See Examples \ref{E:PTypeK} and~\ref{E:PTypeTRP}.
\item
If $d (n) = 2^n$ and $k (n) = 1$ for all~$n,$
then $\af$ is pointwise outer
but does not have the tracial Rokhlin property.
See Examples \ref{E:PTypeTRP} and~\ref{E:PTypeOut}.
\item
There are no actions of this specific type which are
hereditarily saturated but not pointwise outer.
However, the action of $\Zqt^2$ on $M_2$
generated by
$\Ad \left( \begin{smallmatrix} 1 & 0  \\
                                0 & -1 \end{smallmatrix} \right)$
and
$\Ad \left( \begin{smallmatrix} 0 & 1 \\
                                1 & 0 \end{smallmatrix} \right)$
is hereditarily saturated but not pointwise outer.
See Example~\ref{E:4Gp}.
\item
If $d (n) = 2$ for all~$n,$
and $k$ is given by $k (n) = 0$ for $n \geq 2$ and $k (1) = 1,$
then $\af$ is saturated but not hereditarily saturated.
See Examples \ref{E:PTypeSat} and~\ref{E:PTypeHSat}.
\item
If $d (n) = 2$ and $k (n) = 0$ for all~$n,$
then $\af$ is not saturated.
See Example~\ref{E:PTypeSat}.
\end{itemize}

\section{The Rokhlin property and
   K-theoretic freeness}\label{Sec:Rokhlin}

\indent
We treat the Rokhlin property and K-theoretic freeness together
because, in the situations to which they apply well,
they seem roughly comparable in strength.
The usefulness of both is limited to special classes of \ca s.
The Rokhlin property is not useful if there are too few \pj s,
since no action of any nontrivial group on a \uca\  %
with no nontrivial \pj s can have the \rp.
Thus (Example~\ref{E:FreeNotRkh} below),
even if the finite group~$G$ acts freely on~$X,$
the action of $G$ on $C (X)$ need not have the \rp.

The definition of K-theoretic freeness represents
an attempt to turn Theorem~\ref{CT:KThy} into a definition.
The condition must be strengthened;
see the discussion before Definition~\ref{D:KFree} below.
The usefulness of K-theoretic freeness depends on the
presence of nontrivial K-theory:
the trivial action of any finite group on any \ca\  of the
form ${\mathcal{O}}_2 \otimes A$ satisfies the strongest
possible form of K-theoretic freeness.
It is, of course, true that an action on a commutative \uca\  %
is K-theoretically free \ifo\  the action on the corresponding
space is free.
For the product type action of Definition~\ref{D:PType},
K-theoretic freeness is equivalent to the Rokhlin property,
and this seems likely to be true for general
product type actions on UHF~algebras,
and perhaps more generally.
(Example~\ref{E:RPn} limits how far this idea can be taken.)

No known version of noncommutative freeness agrees
both with K-theoretic freeness in the presence of sufficient K-theory
and with the Rokhlin property on~${\mathcal{O}}_2.$
See Problem~\ref{Pr:Free}.
However, for some applications one really needs the Rokhlin property.
See Example~\ref{E:RPn2}.

The Rokhlin property is as in Definition~3.1 of~\cite{Iz1}
(see below);
we first give the equivalent form in
Definition~1.1 of~\cite{PhtRp1}.
The property is, however, much older.
Early uses in \ca s (under a different name)
can be found in \cite{FM}, \cite{HJ1}, and~\cite{HJ2}.
The version for von Neumann algebras appeared even earlier,
for example in~\cite{Jn}.
It is a noncommutative generalization
of the statement of the Rokhlin Lemma in ergodic theory
(for the case $G = \Z$).
The \rp\  can also be considered to be modelled
on Theorem~\ref{CT:Cantor}.
Note, though, that Condition~(\ref{SRPDfn:2}) of the definition
is purely noncommutative in character,
and is essential for the applications of the \rp.
As we will see in Section~\ref{Sec:TRP},
especially Example~\ref{E:TRP} and the following discussion,
the \rp\  is quite rare.

\begin{dfn}\label{SRPDfn}
Let $A$ be a \uca,
and let $\af \colon G \to \Aut (A)$
be an action of a finite group $G$ on $A.$
We say that $\af$ has the
{\emph{Rokhlin property}} if for every finite set
$S \subset A$ and every $\ep > 0,$
there are \mops\  $e_g \in A$ for $g \in G$ such that:
\begin{enumerate}
\item\label{SRPDfn:1}
$\| \af_g (e_h) - e_{g h} \| < \ep$ for all $g, h \in G.$
\item\label{SRPDfn:2}
$\| e_g a - a e_g \| < \ep$ for all $g \in G$ and all $a \in S.$
\item\label{SRPDfn:3}
$\sum_{g \in G} e_g = 1.$
\end{enumerate}
We call the $(e_g)_{g \in G}$ a
{\emph{family of Rokhlin projections}} for $\af,$ $S,$ and $\ep.$
\end{dfn}

\begin{rmk}\label{R:IzumiRP}
The Rokhlin property can be neatly formulated in terms
of central sequence algebras.
Consider the \ca\  of all bounded sequences in $A$
(bounded functions from $\N$ to $A$).
It contains an ideal consisting of all
sequences vanishing at infinity.
Let $A^{\infty}$ be the quotient.
Let $A_{\infty}$ be the relative commutant in $A^{\infty}$
of the embedded copy of $A$ obtained as the image of the set
of constant sequences.
The action $\af \colon G \to \Aut (A)$
induces an action $\af_{\infty} \colon G \to \Aut (A_{\infty}).$
(There is no continuity issue since $G$ is discrete.)
Definition~3.1 of~\cite{Iz1} asks for
\mops\  $e_g \in A_{\infty}$ for $g \in G$ such that
$\sum_{g \in G} e_g = 1$ and such that
$(\af_{\infty})_g (e_h) = e_{g h}$ for all $g, h \in G.$
\end{rmk}

We make a few comments about what happens for more general groups.
For actions of~$\Z,$
one does not consider families of \pj s indexed by~$\Z$
(for which the sum in Definition~\ref{SRPDfn}(\ref{SRPDfn:2})
would not make sense),
but rather families indexed by arbitrarily long
finite intervals in~$\Z.$
(The important points about intervals
are that they are F{\o}lner sets in the group
and that they can tile the group.)
Moreover, to avoid K-theoretic obstructions,
one must in general allow several orthogonal such families
indexed by intervals of different lengths.
See the survey article~\cite{Iz0} and references there.
For the application of the \rp\  to classification of actions of~$\Z,$
see the survey article~\cite{Nk3}.
For more general discrete groups,
one encounters further difficulties with the choice
of subsets of the group.
The group must certainly be amenable,
and the results of~\cite{OW}
suggest that one may only be able to require that the sum
in Definition~\ref{SRPDfn}(\ref{SRPDfn:3}) be close to~$1,$
necessarily in a sense weaker than the norm topology.
The resulting notion looks more like the tracial \rp\  %
(Definition~\ref{D:TRP} below).
In addition, if $G$ is not discrete, one must abandon \pj s.
See~\cite{OW} for actions of locally compact groups on
measure spaces,
see~\cite{Oc} (especially the theorem in Section~6.1)
for actions of countable amenable groups on von Neumann algebras,
see~\cite{Nk0} for actions of $\Z^d$ on \ca s,
see~\cite{Ks2} for actions of~$\R$ on \ca s,
and see~\cite{HW} for actions of compact groups on \ca s.
Much of this, including more on the von Neumann algebra
versions (which came first), is discussed in~\cite{Iz0}.

\begin{exa}\label{E:PTypeRk}
Let $\af \colon \Zqt \to \Aut (A)$ be as in Definition~\ref{D:PType}.
Then $\af$ has the \rp\  \ifo\  $k (n) = \frac{1}{2} d (n)$
for infinitely many $n \in \N.$
This is part of Proposition~2.4 of~\cite{PhtRp4},
but was known long before;
see Lemma~1.6.1 of~\cite{FM}.
\end{exa}

By combining tensor factors,
we can write any such action as a product type action so that
$k (n) = \frac{1}{2} d (n)$ for all $n \in \N.$

It is trivial that if $\af \colon G \to \Aut (C (X))$ has the \rp,
then the action of $G$ on $X$ is free.
As we will see in Corollary~\ref{T:RPImpFree} below,
the \rp\  for an action of a finite group on a unital
type~I \ca\  $A$ implies freeness of the action on $\Prim (A).$
(One can also give a direct proof.)
For totally disconnected~$X,$
the converse is Theorem~\ref{CT:Cantor}.
In general, however, the converse is false.
This is our first encounter with one of the main defects of the \rp,
namely that it is appropriate only for \ca s
with a sufficient supply of \pj s.

\begin{exa}\label{E:FreeNotRkh}
Let $X$ be the circle $S^1$ and let $G = \Zqt,$
with nontrivial element~$g.$
Then $g \zt = - \zt$ generates a free action of $G$ on~$X.$
The corresponding action $\af$ of $G$ on $C (X)$
does not have the \rp,
because $C (X)$ has no nontrivial \pj s.
Specifically, if $\ep < 1$ it is not possible to find
\pj s $e_1, e_g \in C (S^1)$ such that $e_1 + e_g = 1$
and $\| \af_g (e_1) - e_g \| < \ep.$
\end{exa}

Any free action on a connected compact manifold gives
the same outcome.
Other easy examples are
the action of $\Zq{n}$ on the circle $S^1$ by rotation,
and the action of $\Zqt$ on the $m$-sphere $S^m$ via $x \mapsto - x.$

However, there seem to be more subtle issues with the
\rp\  than merely lack of sufficiently many \pj s.

\begin{exa}\label{E:RPn}
In Example~4.1 of~\cite{PhtRp4},
there is an action of $\Zqt$
on a \uca\   $A = \dirlim A_n,$
obtained as the direct limit of actions on the $A_n,$
such that $A_n \cong C (S^{2 k}, \, M_{s (n)})$
for suitable integers $s (0) < s (1) < \cdots$
and a fixed integer $k \geq 1,$
and such that the action on $A_n$ is the tensor product
of the action on $C (S^{2 k})$ induced by $x \mapsto - x$
with an inner action on $M_{s (n)}.$
These actions surely deserve to be considered free.
Therefore the direct limit action
should also deserve to be considered free.
The direct limit algebra $A$ is simple and~AF,
and in particular is well supplied with \pj s
by any reasonable standard.
However, the direct limit action does not have the \rp.
In fact, its crossed product is not~AF,
because the $K_0$-group of the crossed product has torsion
isomorphic to $\Zq{2^k}.$
See Proposition~4.2 of~\cite{PhtRp4} for details.
So the \rp\  would contradict
Theorem~\ref{T:RkStruct}(\ref{CRP_DL:AF}) below.
\end{exa}

We address this issue again in Problem~\ref{Pr:Free} below.

Crossed products by actions with the \rp\  %
preserve many properties of \ca s.
These can be thought of as being related to Theorem~\ref{CT:Mf}.

\begin{thm}\label{T:RkStruct}
Crossed products by actions of finite groups with the Rokhlin property
preserve the following classes of \ca s.
\begin{enumerate}
\item\label{CRP:Simple}
Simple \uca s.
(See Proposition~\ref{T:RPImpSOut} and Theorem~\ref{T:Ks} below.)
\item\label{CRP:DLim}
Various classes of unital but not necessarily simple countable
direct limit \ca s using semiprojective building blocks,
and in which the maps of the direct system need not be injective:
\begin{enumerate}
\item\label{CRP_DL:AF}
AF~algebras.
(See Theorem~2.2 of~\cite{PhtRp1}.)
\item\label{CRP_DL:AI}
AI~algebras.
(See Corollary~3.6(1) of~\cite{OP2}.)
\item\label{CRP_DL:AT}
AT~algebras.
(See Corollary~3.6(2) of~\cite{OP2}.)
\item\label{CRP_DL:NCCW}
Unital direct limits of one dimensional noncommutative CW~complexes.
(See Corollary~3.6(4) of~\cite{OP2}.)
\item\label{CRP_DL:Tp}
Unital direct limits of Toeplitz algebras,
a special case
of the sort studied in~\cite{LS}
except not necessarily of real rank zero.
(See Example~2.10 and Theorem~3.5 of~\cite{OP2}.)
\item\label{CRP_DL:Other}
Various other classes;
see Section~2 and Theorem~3.5 of~\cite{OP2} for details.
\end{enumerate}
\item\label{CRP:SAH}
Simple unital AH~algebras with slow dimension growth and real rank zero.
(See Theorem~3.10 of~\cite{OP2}.)
\item\label{CRP:Abs}
$D$-absorbing separable \uca s for a strongly self-absorbing
\ca~$D.$
(See Theorem~1.1(1) and Corollary~3.4(i) of~\cite{HW}. 
See~\cite{HW} for the definition of a strongly self-absorbing \ca.)
\item\label{CRP:RR0}
Unital \ca s with real rank zero.
(See Proposition~4.1(1) of~\cite{OP2}.)
\item\label{CRP:tsr1}
Unital \ca s with stable rank one.
(See Proposition~4.1(2) of~\cite{OP2}.)
\item\label{CRP:UCT}
Separable nuclear \uca s whose quotients all satisfy the
Universal Coefficient Theorem.
(See Proposition~3.7 of~\cite{OP2}.)
\item\label{CRP:Kbg}
Unital Kirchberg algebras
(simple separable nuclear purely infinite \ca s)
which satisfy the Universal Coefficient Theorem.
(See Corollary~3.11 of~\cite{OP2}.)
\item\label{CRP:ADiv}
Separable unital approximately divisible \ca s.
(See Corollary~3.4(2) of \cite{HW},
which also covers actions of compact groups;
also see Proposition~4.5 of~\cite{OP2}.)
\item\label{CRP:PjProp}
Unital \ca s with the ideal property and \uca s with
the projection property.
(See \cite{PP};
also see~\cite{PP} for the definitions of these properties).
\item\label{CRP:KThy}
Simple \uca s whose K-theory:
\begin{enumerate}
\item\label{CRP:KThy:1}
Is torsion free.
\item\label{CRP:KThy:2}
Is a torsion group.
\item\label{CRP:KThy:3}
Is zero.
\end{enumerate}
(See the discussion below.)
\end{enumerate}
\end{thm}

Further classes will appear in~\cite{PP}.

Many of the parts of Theorem~\ref{T:RkStruct} are special
for actions of finite (or compact) groups.
For example,
Parts (\ref{CRP_DL:AF}) and~(\ref{CRP_DL:AI}) fail for $G = \Z,$
because the \cp\  will have nontrivial $K_1$-group.
We refer to Section~2.2 of~\cite{Iz0} and the references there
for many positive results for actions of $\Z$ with the \rp.

The main ingredient for Theorem~\ref{T:RkStruct}(\ref{CRP:KThy})
is the following result of Izumi.
Using it, one can derive many other statements similar to
the ones given.

\begin{thm}[Theorem~3.13 of~\cite{Iz1}]\label{T:RPKthInj}
Let $A$ be a simple \uca,
let $G$ be a finite group,
and let $\af \colon G \to \Aut (A)$ be an action with the \rp.
Then the inclusion $A^G \to A$
induces an injective map $K_* (A^G) \to K_* (A).$
\end{thm}

The other ingredient is that the \rp\  implies that
$K_* (C^* (G, A, \af)) \cong K_* (A^G).$
Proposition~\ref{P:RPImpHS} below shows that the \rp\  implies
hereditary saturation (Definition~\ref{D:HSat} below),
and in particular implies saturation (Definition~\ref{D:Sat} below).
Now combine Proposition~7.1.8 and Theorem~2.6.1 of~\cite{Ph1}.

For some of the classes in Theorem~\ref{T:RkStruct},
such as those
in Parts~(\ref{CRP:SAH}), (\ref{CRP:RR0}), (\ref{CRP:tsr1}),
(\ref{CRP:UCT}), and~(\ref{CRP:PjProp}),
it is expected that weaker conditions than the \rp\  should suffice.
This is certainly true for Part~(\ref{CRP:Simple}).
See Theorem~\ref{T:Ks} and Corollary~\ref{T:StrConnesAndSimp} below.

Theorem~\ref{T:RkStruct}(\ref{CRP:Abs}) holds for actions
of second countable compact groups with the \rp\   and,
provided the strongly selfabsorbing \ca\  is $K_1$-injective,
for actions
of $\Z$ and $\R$ with the \rp.
See Theorem~1.1(1) of~\cite{HW}.

Theorem~\ref{T:RkStruct}(\ref{CRP:KThy})
brings back the point that the Rokhlin property is slightly
too strong.
For example,
if $G$ is a finite group acting freely on a compact
space $X$ such that $K^* (X)$ is torsion free,
it does not follow that $K_* (C^* (G, X))$ is torsion free.
The action of $\Zqt$ on $S^2$ generated by $x \mapsto - x$
is a counterexample.
(The K-theory of the \cp\  is the same as for the
real projective space $\R P^2.$)
Example~\ref{E:RPn} shows that this phenomenon
can even occur for an action on a simple AF~algebra.

We also have:

\begin{thm}[Theorem~1.25 of~\cite{Sr}]\label{T:RokhIdeals}
Let $\af \colon G \to \Aut (A)$
be an action of a finite group on a \uca\  which has the \rp.
Then every ideal $J \subset C^* (G, A, \af)$ has the form
$C^* (G, I, \af_{(\cdot)} |_I)$
for some $G$-invariant ideal $I \subset A.$
\end{thm}

For an application of these structural results,
see~\cite{PV}.
This paper uses a crossed product construction
(following Connes' von Neumann algebra construction~\cite{Connes1})
to produce a simple separable exact \ca\  $A$
which is not isomorphic to its opposite algebra,
and which has a number of nice properties.
The action has the \rp,
and this fact is crucial in the computation of $K_* (A)$
and the verification of a number of the properties of~$A.$

The use of the \rp\  to obtain structural results
for \cp s by finite groups seems to be very recent
(although it has a longer history for actions of~$\Z$).
Indeed, Theorem~\ref{T:RkStruct}(\ref{CRP_DL:AF})
(crossed products
by finite group actions with the \rp\  %
preserve the class of AF~algebras)
could easily have been proved long ago.
It was actually proved only after
the analogous statement
for the tracial \rp\  and \ca s with tracial rank zero,
Theorem~\ref{T:RokhTAF}
(although both theorems appear in the same paper).
The main reason is that the tracial analog
seems to be much more useful for structure theory.
See Section~\ref{Sec:TRP}.

The main application so far of the \rp\  has been
to the structure and classification of group actions.
This project began in von Neumann algebras.
As just one example,
Jones proved (Corollary~5.3.7 of~\cite{Jn})
that {\emph{every}} pointwise outer action
of a finite group~$G$ on the hyperfinite factor $R$
of type~II$_1$ is conjugate to a particular model action
of~$G.$
That is, up to conjugacy,
there is only one pointwise outer action of $G$ on~$R.$
An essential step in the proof is showing that
pointwise outerness implies the von Neumann algebra
analog of the \rp.
Ocneanu~\cite{Oc} extended the result to actions of general
countable amenable groups.
Classification of actions
on \ca s with the \rp\  is the main thrust
of the papers~\cite{FM}, \cite{HJ1}, and~\cite{HJ2},
and the more recent papers \cite{Iz1} and~\cite{Iz2}.

We state four theorems from~\cite{Iz2}.
The first two require the following ``model action''.

\begin{dfn}\label{D:RegPType}
Let $G$ be a finite group, and let $n = \card (G).$
Let $u \colon G \to M_n$ be the image of
the regular representation of~$G$
under some isomorphism $L (l^2 (G)) \to M_n.$
Let $\ld \colon G \to \Aut (M_n)$
be $\ld_g (a) = u (g) a u (g)^*.$
Let $D_G$ be the UHF~algebra
$D_G = \bigotimes_{m = 1}^{\infty} M_{n},$
and let $\mu^G \colon G \to \Aut (D_G)$
be the product type action
$\mu_g^G = \bigotimes_{m = 1}^{\infty} \ld_g.$
\end{dfn}

For $G = \Zqt,$
the action $\mu^G$ is obtained as in Definition~\ref{D:PType}
by taking $d (n) = 2$ and $k (n) = 1$ for all~$n.$

\begin{thm}[Theorem~3.4 of~\cite{Iz2}]\label{T:IzPIStab}
Let $A$ be a unital Kirchberg algebra
(simple, separable, purely infinite, and nuclear)
which satisfies the Universal Coefficient Theorem.
Let $G$ be a finite group and let $\af \colon G \to \Aut (A)$
have the Rokhlin property.
Suppose $(\af_g)_*$ is the identity on $K_* (A)$ for all $g \in G.$
Then $\af$ is conjugate to the action $g \mapsto \id_A \otimes \mu_g^G$
on $A \otimes D_G.$
\end{thm}

\begin{thm}[Theorem~3.5 of~\cite{Iz2}]\label{T:IzTR0Stab}
Theorem~\ref{T:IzPIStab} remains true if instead
$A$ is a simple separable \uca\  with tracial rank zero
which satisfies the Universal Coefficient Theorem.
\end{thm}

\begin{thm}[Theorem~4.2 of~\cite{Iz2}]\label{T:IzPIConj}
Let $A$ be a unital Kirchberg algebra
which satisfies the Universal Coefficient Theorem.
Let $G$ be a finite group.
Let $\af, \bt \colon G \to \Aut (A)$ have the Rokhlin property.
Suppose $(\af_g)_* = (\bt_g)_*$ on $K_* (A)$ for all $g \in G.$
Then there exists $\te \in \Aut (A)$ such that
$\bt_g = \te \circ \af_g \circ \te^{-1}$ for all $g \in G,$
and such that $\te_*$ is the identity on $K_* (A).$
\end{thm}

\begin{thm}[Theorem~4.3 of~\cite{Iz2}]\label{T:IzTR0Conj}
Theorem~\ref{T:IzPIConj} remains true if instead
$A$ is a simple separable \uca\  with tracial rank zero
which satisfies the Universal Coefficient Theorem.
\end{thm}

See~\cite{Nk3} for results on classification of actions
of~$\Z$ with the \rp.

We should point out that some work has been done on
classification of finite and compact group actions
without the \rp,
assuming instead that the action is compatible
with a direct limit realization of the algebra,
and also that the actions on the algebras in the direct system
have special forms.
See, for example, \cite{HR1}, \cite{HR2}, \cite{Kso}, and~\cite{Su}.

The \rp\  enjoys the following permanence properties.
Proposition~\ref{T:RPerm}(\ref{T:RPerm:Ideal})
is the analog of the fact that the restriction of a free action to
a closed subset of the space is still free,
(\ref{T:RPerm:Sbgp})~is the analog of the fact that the restriction
of a free action to
a subgroup is still free,
(\ref{T:RPerm:Tens})~is the analog of the fact that
a diagonal action on a product is free if one of the factors is free,
and (\ref{T:RPerm:DLim})~is the analog of the fact that
an equivariant inverse limit of free actions is free.

\begin{prp}\label{T:RPerm}
Let $A$ be a \uca, let $G$ be a finite group,
and let $\af \colon G \to \Aut (A)$ have the \rp.
Then:
\begin{enumerate}
\item\label{T:RPerm:Ideal}
If $I \subset A$ is a $\af$-invariant ideal,
then the induced action ${\overline{\af}}$ of $G$ on $A / I$
has the \rp.
\item\label{T:RPerm:Sbgp}
If $H \subset G$ is a subgroup,
then $\af |_H$ has the \rp.
\item\label{T:RPerm:Tens}
If $\bt \colon G \to \Aut (B)$ is any action of $G$ on a \uca,
then the tensor product action
$g \mapsto \af_g \otimes \bt_g$ of $G$ on $\Aut (A \otimes B),$
for any C*~tensor product on which it is defined, has the \rp.
\setcounter{TmpEnumi}{\value{enumi}}
\end{enumerate}
In addition:
\begin{enumerate}
\setcounter{enumi}{\value{TmpEnumi}}
\item\label{T:RPerm:DLim}
If $A = \dirlim A_n$ is a direct limit of \ca s with unital maps,
and $\af \colon G \to \Aut (A)$ is an action obtained
as the direct limit of actions $\af^{(n)} \colon G \to \Aut (A_n),$
such that $\af^{(n)}$ has the \rp\  for all $n,$
then $\af$ has the \rp.
\end{enumerate}
\end{prp}

\begin{proof}
For~(\ref{T:RPerm:Ideal}),
let $\pi \colon A \to A / I$ be the quotient map.
Let $S \subset A / I$ be finite and let $\ep > 0.$
Choose a finite set $T \subset A$ such that $\pi (T)$ contains $S.$
Let $(e_g)_{g \in G}$ be a
family of Rokhlin projections for $\af,$ $T,$ and~$\ep.$
Then $( \pi (e_g))_{g \in G}$ is a
family of Rokhlin projections for $\af,$ $S,$ and~$\ep.$

For~(\ref{T:RPerm:Sbgp}),
set $n = \card (G/H).$
Let $S \subset A$ be finite and let $\ep > 0.$
Choose a system $C$ of right coset representatives of $H$ in $G.$
Let $(f_g)_{g \in G}$ be a
family of Rokhlin projections for $\af,$ $S,$ and $\ep / n.$
Then use the \pj s
$e_h = \sum_{c \in C} f_{h c}$ for $h \in H.$

For~(\ref{T:RPerm:Tens}),
one first checks that in Definition~\ref{SRPDfn}
it suffices to fix a subset $R \subset A$ which generates $A$ as a \ca,
and verify the condition of the definition only for finite subsets
$S \subset R.$
In this case, we take
\[
R = \{ a \otimes b \colon {\mbox{$a \in A$ and  $b \in B$}} \}.
\]
We can further restrict to finite subsets of the form
\[
Q = \{ a \otimes b \colon {\mbox{$a \in S$ and  $b \in T$}} \},
\]
for finite sets $S \subset A$ and $T \subset B.$
Set $M = \sup_{b \in T} \| b \|.$
Let $(e_g)_{g \in G}$ be a
family of Rokhlin projections for $\af,$ $S,$ and $\ep / M.$
Then $( e_g \otimes 1)_{g \in G}$ is a
family of Rokhlin projections for $\af \otimes \bt,$ $Q,$ and $\ep.$

To prove~(\ref{T:RPerm:DLim}),
first use~(\ref{T:RPerm:Ideal}) to replace each $A_n$ with its
quotient by the kernel of the map $A_n \to A.$
Thus we may assume that $A = {\overline{\bigcup_{n = 0}^{\infty} A_n}}.$
As in the previous part, we restrict to a generating set,
here $R = \bigcup_{n = 0}^{\infty} A_n.$
So let $S \subset R$ be finite and let $\ep > 0.$
Choose $n$ such that $S \subset A_n.$
Then a family of Rokhlin projections for $\af^{(n)},$ $S,$ and~$\ep$
is a family of Rokhlin projections for $\af,$ $S,$ and~$\ep.$
\end{proof}

We can't reasonably talk about the Rokhlin property passing to
invariant ideals
without a definition of the \rp\  for non\uca s.
We know of no definition in the literature, but it
seems reasonable to simply take the \pj s to be in the
multiplier algebra.
An action of $G$ on a \ca~$A$ always extends to
an action $g \mapsto M (\af)_g$ on the multiplier algebra $M (A).$
In general, one only gets continuity
of $g \mapsto M (\af)_g (a)$ in the strict topology,
but this is irrelevant for finite groups.

\begin{dfn}\label{D:MRP}
Let $A$ be a not necessarily \uca,
and let $\af \colon G \to \Aut (A)$
be an action of a finite group $G$ on $A.$
We say that $\af$ has the
{\emph{multiplier Rokhlin property}} if for every finite set
$S \subset A$ and every $\ep > 0,$
there are \mops\  $e_g \in M (A)$ for $g \in G$ such that:
\begin{enumerate}
\item\label{D:MRP:1}
$\| M (\af)_g (e_h) - e_{g h} \| < \ep$ for all $g, h \in G.$
\item\label{D:MRP:2}
$\| e_g a - a e_g \| < \ep$ for all $g \in G$ and all $a \in S.$
\item\label{D:MRP:3}
$\sum_{g \in G} e_g = 1.$
\end{enumerate}
\end{dfn}

\begin{prp}\label{P:RokhIdeal}
Let $A$ be a \uca,
and let $\af \colon G \to \Aut (A)$
be an action of a finite group $G$ on $A.$
Let $J \subset A$ be a $G$-invariant ideal.
Suppose $\af$ has the \rp.
Then $\af_{(\cdot)} |_J$ has the multiplier \rp.
\end{prp}

\begin{proof}
Let $F \subset J$ be a finite set, and let $\ep > 0.$
Choose a family $(e_g)_{g \in G}$
of Rokhlin projections for $\af,$ $S,$ and~$\ep.$
Let $\ph \colon A \to M (J)$ be the standard \hm\  obtained from the
fact that $J$ is an ideal in~$A.$
Then $(\ph (e_g) )_{g \in G}$ satisfies
the condition of Definition~\ref{D:MRP}
for the given choices of $F$ and~$\ep.$
\end{proof}

Unfortunately,
and this is related to the issues with the \rp\  already discussed,
there is no theorem for extensions.

\begin{exa}\label{E:Rokh2of3}
Let $G = \Zqt.$
Let $\af \colon G \to \Aut (C (S^1))$
be as in Example~\ref{E:FreeNotRkh}.
(Recall that this action comes from the action on $S^1$
generated by $\zt \mapsto - \zt.$)
Let $J = C_0 (S^1 \setminus \{ \pm 1 \} ),$
which is an ideal in $C (S^1).$
It is easily checked that $\af_{(\cdot)} |_J$ has the multiplier \rp\  %
and that the induced action on $C (S^1) / J \cong \C^2$ has the \rp.
As in Example~\ref{E:FreeNotRkh}, however,
$\af$ itself does not have the \rp.
\end{exa}

We now turn to K-theoretic freeness.
As Theorem~\ref{CT:KThy} shows,
K-theory gives a neat condition for an action of a compact Lie
group
on a compact space to be free.
The book~\cite{Ph1} mostly describes the effort to turn
this result into a definition of freeness of actions on \ca s.
Also see the survey~\cite{Ph0}.
This idea is useless if there is no K-theory.
For example,
the trivial action
on any \ca\  of the form ${\mathcal{O}}_2 \otimes A$
is K-theoretically free.
Moreover, unlike for actions on $C (X),$
for \ca s the condition of Theorem~\ref{CT:KThy}
passes neither to quotients by invariant ideals
nor to subgroups,
so it is necessary to build these features
into the noncommutative definition.
(Example~4.1.7 of~\cite{Ph1} implies that the property
does not pass to quotients by invariant ideals,
although it is stated in terms of the ideal rather than the quotient.
Example~\ref{E:939} below gives an action of $\Zq{4}$
on a UHF~algebra
which satisfies the condition of Theorem~\ref{CT:KThy},
but such that the restriction of the action to
the subgroup $\Zqt \S \Zq{4}$ does not;
the restriction is in fact inner.)
Subject to these caveats,
the definition permits some interesting results,
such as for \ca s of type~I
and for product type actions on UHF~algebras.
It covers $\zt \mapsto - \zt$ on $S^1,$
unlike the Rokhlin property.

Equivariant K-theory
for an action of a compact group $G$ on a \ca\  $A,$
denoted $K_*^G (A),$ is as in Chapter~2 of~\cite{Ph1}.
Like the equivariant K-theory of a space,
as discussed before Theorem~\ref{CT:KThy},
it is a module over the representation ring $R (G),$
and for a prime ideal $P \S R (G),$
the localization $K_*^G (A)_P$ is defined.
The following is from Definitions 4.1.1, 4.2.1, and~4.2.4,
and the discussion after Definition~4.2.1, in~\cite{Ph1}.
The term ``locally discrete K-theory'' refers to the
$I (G)$-adic topology on $K_*^G (A)$;
see discussion after Definition~4.1.1 in~\cite{Ph1}.

\begin{dfn}\label{D:KFree}
Let $A$ be a \ca,
and let $\af \colon G \to \Aut (A)$
be an action of a finite group $G$ on~$A.$
We say that $\af$ has
{\emph{locally discrete K-theory}}
if for every prime ideal $P$
in the representation ring $R (G)$ which does not contain $I (G),$
the localization $K_*^G (A)_P$ is zero.
We say that $\af$ is {\emph{K-free}}
if for every invariant ideal $I \subset A,$
the induced action on $A / I$ has locally discrete K-theory.
We say that $\af$ is {\emph{totally K-free}}
if for every subgroup $H \subset G,$
the restricted action $\af |_H$ is K-free.
\end{dfn}

Other related conditions,
including ones involving equivariant KK-theory,
and many more results and examples than can be discussed here,
can be found in~\cite{Ph1}.
Also see the survey~\cite{Ph0}.

The definition generalizes reasonably to actions of compact Lie groups,
but not to actions of noncompact groups.

The definition behaves well for type~I \ca s:

\begin{thm}[Theorems 8.1.4 and~8.2.5 of~\cite{Ph1}]\label{T:TKF-Type1}
Let $A$ be a separable type~I \ca,
and let $\af \colon G \to \Aut (A)$
be an action of a compact Lie group $G$ on~$A.$
Then $\af$ is totally K-free \ifo\  %
the induced action of $G$ on $\Prim (A)$ is free.
\end{thm}

In fact, one direction holds for any \ca;
Example~\ref{E:FreeNotRkh} shows that the corresponding
statement for the \rp\  fails.

\begin{thm}[Theorem~4.3.8 of~\cite{Ph1}]\label{T:TKF-Prim}
Let $A$ be a \ca,
and let $\af \colon G \to \Aut (A)$
be an action of a finite group $G$ on~$A.$
Suppose that the induced action of $G$ on $\Prim (A)$ is free.
Then $\af$ is totally K-free.
\end{thm}

For our standard product type action,
total K-freeness is equivalent to the \rp:

\begin{exa}\label{E:PTypeK}
Let $\af \colon \Zqt \to \Aut (A)$ be as in Definition~\ref{D:PType}.
Then \tfae:
\begin{enumerate}
\item\label{E:PTypeK:TKt}
$\af$ is totally K-free.
\item\label{E:PTypeK:LD}
$\af$ has locally discrete K-theory.
\item\label{E:PTypeK:IG}
$I (\Zqt) K_0^{\Zqt} (A) = 0.$
\item\label{E:PTypeK:Half}
$k (n) = \frac{1}{2} d (n)$ for infinitely many $n \in \N.$
\item\label{E:PTypeK:RP}
$\af$ has the \rp.
\end{enumerate}
This follows from Theorems~9.2.4 and~9.2.6 of~\cite{Ph1}
and Proposition~2.4 of~\cite{PhtRp4}.
(These results also give a number of other equivalent conditions.)
\end{exa}

In this example,
the equivalence of Condition~(\ref{E:PTypeK:LD})
with the rest is slightly misleading.
The action of $\Zqt \times \Zqt$ on $M_2$ in Example~\ref{E:4Gp}
below is K-free but not totally K-free
(see Example~4.2.3 of~\cite{Ph1}),
and certainly does not have the \rp.
This kind of thing can happen even for cyclic groups:

\begin{exa}\label{E:939}
Example 9.3.9 of~\cite{Ph1} gives an action of $\Zq{4}$
on the $2^{\infty}$ UHF~algebra
(see Remark 9.3.10 of~\cite{Ph1} for the identification of the algebra)
which is a direct limit of actions on \fd\  \ca s,
is K-free, but is not totally K-free.
It is observed there that the automorphism corresponding
to the order two element of $\Zq{4}$ is inner.
Therefore Proposition~\ref{T:RPImpSOut} below implies that
this action does not have the \rp.

The crossed product is simple,
in fact, isomorphic to the $2^{\infty}$ UHF~algebra,
by Remark 9.3.10 of~\cite{Ph1}.
\end{exa}

The \rp\  implies total K-freeness in complete generality;
in fact, it implies that $I (G) K_*^G (A) = 0.$

\begin{prp}\label{T:RPImpDKth}
Let $A$ be a \uca, let $G$ be a finite group,
and let $\af \colon G \to \Aut (A)$ have the \rp.
Then $I (G) K_*^G (A) = 0.$
\end{prp}

\begin{proof}
We start with some reductions.
First, replacing $A$ by $C (S^1) \otimes A$
with $G$ acting trivially on $S^1,$
and using Proposition~\ref{T:RPerm}(\ref{T:RPerm:Tens})
and
$K_0^G (C (S^1) \otimes A) \cong K_0^G (A) \oplus K_1^G (A),$
we see that it suffices to show that $I (G) K_0^G (A) = 0.$
Next, by Corollary~2.4.5 of~\cite{Ph1},
we need only consider elements of $K_0^G (A)$ which
are represented by invariant \pj s in $L (W) \otimes A$
for some \fd\  unitary representation space $W$ of~$G.$
Again using Proposition~\ref{T:RPerm}(\ref{T:RPerm:Tens}),
we may replace $A$ by $L (W) \otimes A,$
and thus consider the class $[p]$ of an invariant \pj\  in $A.$
Next, we need only consider elements of $I (G)$ which span $I (G),$
so it suffices to prove that
if $V$ is a \fd\  unitary representation space of $G,$
with representation $g \mapsto u_g,$
and if $V_0$ is the same space with the trivial action of $G,$
then $([V] - [V_0]) [p] = 0$ in $K_0 (G).$
Now let $\ph, \ph_0 \colon A \to L (V_0 \oplus V) \otimes A$ be the maps
defined by $\ph_0 (a) = (1_{V_0} \oplus 0) \otimes a$
and $\ph (a) = (0 \oplus 1_V) \otimes a.$
By Remark~2.4.6 of~\cite{Ph1},
it suffices to find
a $G$-invariant element $y \in L (V_0 \oplus V) \otimes A$
such that $y^* y = \ph_0 (p)$ and $y y^* = \ph (p).$
Note that the action of $G$ on $L (V_0 \oplus V) \otimes A$
is given by
$g \mapsto {\overline{\af}}_g = \Ad (1 \oplus u_g) \otimes \af_g.$

Set $n = \card (G)$ and $\ep = 1 / (n^2 + 2 n).$
Let $(e_g)_{g \in G}$ be a
family of Rokhlin projections for $\af,$ $\{ p \},$ and $\ep.$
Set $w_0 = \sum_{h \in G} u_h \otimes e_h,$
which is a unitary in $L (V) \otimes A.$
Since $V_0 = V$ as vector spaces, we can identify
$L (V_0 \oplus V) \otimes A$ with $M_2 (L (V) \otimes A),$
and define $w \in L (V_0 \oplus V) \otimes A$ by
\[
w = \left( \begin{array}{cc}
  0     &  w_0^*        \\
  w_0   &  0
\end{array} \right).
\]

In matrix form,
we have
\[
w \ph_0 (p) w^*
  = \left( \begin{array}{cc}
  0     &  w_0^*        \\
  w_0   &  0
\end{array} \right)
\left( \begin{array}{cc}
  1_V \otimes p     & 0        \\
  0                 &  0
\end{array} \right)
\left( \begin{array}{cc}
  0     &  w_0^*        \\
  w_0   &  0
\end{array} \right)
= \left( \begin{array}{cc}
  0     &  0       \\
  0     &  w_0 (1_V \otimes p) w_0^*
\end{array} \right).
\]
Using orthogonality of the \pj s $e_g,$ we have
\begin{align*}
\big\| w_0 (1_V \otimes p) w_0^* - 1_V \otimes p \big\|
& = \big\| w_0 (1_V \otimes p) w_0^* - w_0 w_0^* (1_V \otimes p) \big\|
             \\
& \leq \sum_{g, h \in G}
   \big\| u_{g h^{-1}} \otimes e_g p e_h
                     - u_{g h^{-1}} \otimes e_g e_h p \big\|
              \\
& \leq n \sum_{h \in G} \| p e_h - e_h p \|
  < n^2 \ep.
\end{align*}
Thus,
$\| w \ph_0 (p) w^* - \ph (p) \| < n^2 \ep.$
With
$x_g = \sum_{h \in G} u_{g h} \otimes \af_g (e_h),$
a calculation shows that
\[
{\overline{\af}}_g (w)
 = \left( \begin{array}{cc}
  0     &  x_g^*        \\
  x_g   &  0
\end{array} \right).
\]
Therefore
\[
\| {\overline{\af}}_g (w) - w \|
  = \| x_g - w_0 \|
  \leq \sum_{h \in G} \| \af_g (e_h) - e_{g h} \|
  < n \ep.
\]
It follows that for $g, h \in G,$ we have
\begin{align*}
&
\big\| \ph (p) {\overline{\af}}_g (w)
           \ph_0 (p) {\overline{\af}}_h (w)^* \ph (p)
      - \ph (p) \big\|   \\
& \hspace*{3em} {\mbox{}}
  \leq \| {\overline{\af}}_g (w) - w \|
         + \| w \ph_0 (p) w^* - \ph (p) \|
         + \| {\overline{\af}}_h (w) - w \|
  < (n^2 + 2 n) \ep
  = 1.
\end{align*}
Similarly,
\[
\big\| \ph_0 (p) {\overline{\af}}_g (w)^*
           \ph (p) {\overline{\af}}_h (w) \ph_0 (p)
      - \ph_0 (p) \big\|
  < 1.
\]

Now set
\[
b = \frac{1}{n} \sum_{g \in G} \ph (p) {\overline{\af}}_g (w) \ph_0 (p),
\]
which is $G$-invariant.
Using the estimates in the previous paragraph,
we get
\[
\| b^* b - \ph_0 (p) \| < 1
\andeqn
\| b b^* - \ph (p) \| < 1.
\]
Therefore, with functional calculus evaluated in
$\ph_0 (p) \big[ L (V_0 \oplus V) \otimes A \big] \ph_0 (p),$
the $G$-invariant element
$y = b (b^* b)^{-1/2}$ satisfies
$y^* y = \ph_0 (p)$ and $y y^* = \ph (p).$
\end{proof}

\begin{cor}\label{C:RkhImpTKF}
Let $A$ be a \uca, let $G$ be a finite group,
and let $\af \colon G \to \Aut (A)$ have the \rp.
Then $\af$ is totally K-free.
\end{cor}

\begin{proof}
By Lemma~4.1.4 of~\cite{Ph1},
it suffices to prove that for every subgroup $H \subset G$
and every $H$-invariant ideal $I \subset A,$
the action $\bt$ of $H$ on $A / I$ induced by $\af |_H$
has locally discrete K-theory.
Now $\bt$ has the \rp\  %
by parts (\ref{T:RPerm:Ideal}) and~(\ref{T:RPerm:Sbgp})
of Proposition~\ref{T:RPerm}.
So $I (H) K_*^H (A / I) = 0$ by Proposition~\ref{T:RPImpDKth}.
The conclusion now follows from Proposition~4.1.3 of~\cite{Ph1}.
\end{proof}

We return to the relationship between (total) K-freeness and the
\rp\  after discussing permanence properties.

\begin{prp}\label{T:TKFPerm}
Let $A$ be a \uca, let $G$ be a finite group,
and let $\af \colon G \to \Aut (A)$ be an action of $G$ on~$A.$
\begin{enumerate}
\item\label{T:TKFPerm:IffExt}
If $I \subset A$ is a $\af$-invariant ideal,
then $\af$ is totally K-free \ifo\  %
$\af_{( \cdot )} |_I$
and the induced action of $G$ on $A / I$
are both totally K-free.
\item\label{T:TKFPerm:Sbgp}
If $\af$ is totally K-free,
and if $H \subset G$ is a subgroup,
then $\af |_H$ is totally K-free.
\item\label{T:TKFPerm:DLim}
If $A = \dirlim A_n$ is a direct limit of \ca s,
and $\af \colon G \to \Aut (A)$ is an action obtained
as the direct limit of actions $\af^{(n)} \colon G \to \Aut (A_n),$
such that $\af^{(n)}$ is totally K-free for all~$n,$
then $\af$ is totally K-free.
\end{enumerate}
\end{prp}

\begin{proof}
Part~(\ref{T:TKFPerm:IffExt}) is built into the definition;
see Proposition~4.2.6 of~\cite{Ph1}.
Part~(\ref{T:TKFPerm:Sbgp}) is also built into the definition.

Part~(\ref{T:TKFPerm:DLim}) was overlooked in~\cite{Ph1}.
For locally discrete K-theory, it is Lemma 4.2.14 of~\cite{Ph1}.
Let $\ph_n \colon A_n \to A$ be the maps obtained from
the direct limit realization of~$A.$
Let $H \subset G$ be a subgroup,
and let $J \subset A$ be an $H$-invariant ideal.
Set $J_n = \ph_n^{-1} (J),$ which is an $H$-invariant ideal in~$A_n.$
Then $J = \dirlim J_n.$
For each $n,$ the restriction $\big( \af^{(n)} |_H \big) |_{J_n}$
has locally discrete K-theory.
Therefore so does $\big( \af |_H \big) |_{J}.$
\end{proof}

One can also say a limited amount about actions on tensor products.
See Section~6.6 of~\cite{Ph1}.

We now come back to the relationship with the \rp.

\begin{exa}\label{E:RPn2}
Let $\af \colon \Zqt \to A$ be the action of Example~\ref{E:RPn}.
The \ca~$A$ is a simple AF~algebra.
The construction of~$\af,$
together with Theorem~\ref{T:TKF-Type1} and
Proposition~\ref{T:TKFPerm}(\ref{T:TKFPerm:DLim}),
imply that $\af$ is totally K-free.
However,
we saw in Example~\ref{E:RPn} that $\af$ does not have the \rp\  %
and that $C^* (\Zqt, A, \af)$ is not~AF.
In particular,
the crossed product of a simple AF~algebra by a totally K-free action
need not be~AF.
In fact,
Parts~(\ref{CRP_DL:AF}), (\ref{CRP_DL:AI}), and~(\ref{CRP_DL:AT})
of Theorem~\ref{T:RkStruct} all fail
with total K-freeness in place of the \rp.
\end{exa}

It seems to us that the fault is again with the \rp.
The conclusion of Proposition~\ref{T:RPImpDKth}
stronger than ought to hold for a version of noncommutative freeness.
Indeed, 
there is a free action of a finite group $G$
on a compact metric space~$X$
such that $I (G) K_* (C (X)) \neq 0.$
For example, the actions of $G = \Zqt$ on $S^{2 n}$ and $S^{2 n + 1}$
generated by $x \mapsto - x$
give
\[
K^0_G (S^{2 n}) \cong K^0_G (S^{2 n + 1}) \cong R (G) / I (G)^{n + 1},
\]
as can be seen from Corollary~2.7.5 and the discussion after
Corollary~2.7.6 in~\cite{At}.
It then also follows
(using the proof of Proposition~4.2 of~\cite{PhtRp4})
that suitable choices for $A$ in Example~\ref{E:RPn}
give $I (G) K_0^G (A) \neq 0.$

\begin{pbm}\label{Pb:KFTImpRkh}
Suppose that $A$ is a unital AF~algebra,
$G$ is a finite group,
and $\af \colon G \to \Aut (A)$ is a totally K-free action
which is a direct limit of actions on finite dimensional \ca s.
Does it follow that $\af$ has the \rp?
\end{pbm}

One might want to assume that $A$ is simple,
or even that $\af$ is a product type action.
For product type actions of $\Zqt,$
this is contained in Example~\ref{E:PTypeK}.

It seems that the \rp\  and (total) K-freeness attempt to
detect, not quite successfully,
a strong version of freeness of actions of finite groups on \ca s,
something which, to borrow a suggestion from Claude Schochet,
might be called a ``noncommutative covering space".
The \rp\  is too strong,
even apart from the existence of \pj s,
as is shown by Example~\ref{E:RPn},
while conditions involving K-theory are too weak when there is no
K-theory.
The following problem thus seems interesting,
even though it is not clear what applications it might have.

\begin{pbm}\label{Pr:Free}
Find a well behaved version of freeness of finite group actions
on \uca s
which agrees with total K-freeness for actions on
AF~algebras and type~I \ca s,
and agrees with the \rp\  for actions on
the Cuntz algebra~${\mathcal{O}}_2.$
\end{pbm}

One would hope for the following:
\begin{enumerate}
\item\label{Pr:Free:Perm}
The condition should pass to invariant ideals,
to quotients by invariant ideals,
and to subgroups.
It should also be preserved under extensions.
\item\label{Pr:Free:Tensor}
The condition should be preserved when taking
tensor products with arbitrary actions
(with an arbitrary tensor norm such that the action
extends to the tensor product).
\item\label{Pr:Free:AH}
The condition should be equivalent to total K-freeness
for direct limit actions on AH~algebras.
\item\label{Pr:Free:KFree}
The condition should imply total K-freeness for general \uca s.
\item\label{Pr:Free:Rokh}
The Rokhlin property should imply the condition
for general \uca s.
\item\label{Pr:Free:Type1}
For type~I \ca s,
the condition should be equivalent to free action
on the primitive ideal space.
\item\label{Pr:Free:Outer}
The condition should imply strong pointwise outerness
(Definition~\ref{D:SPOut} below) for arbitrary \uca s.
\end{enumerate}

Example~\ref{E:RPn2} shows that the freeness condition we are
asking for should not imply the \rp\  for actions on UHF algebras.
It is not actually clear that the right condition should
imply the \rp\  for actions on~${\mathcal{O}}_2.$
As evidence that a difference between behavior
on UHF algebras and on~${\mathcal{O}}_2$
should be expected,
consider the tensor flip~$\ph_A,$
the action of $\Zq{2}$ on $A \otimes_{\mathrm{min}} A$
generated by $a \otimes b \mapsto b \otimes a.$
If $A$ is a UHF algebra,
then $\ph_A$ never has the \rp\  (\cite{OP3})
and is never K-free.
(These statements follow easily from Example~\ref{E:PTypeK}.)
However, $\ph_{{\mathcal{O}}_2}$ does have the \rp,
by Example~5.2 of~\cite{Iz1}.

\section{The tracial Rokhlin property and outerness
  in factor representations}\label{Sec:TRP}

\indent
The \trp\  is a weakening of the \rp,
and which is much more common.
Unfortunately,
for now we only know the right definition of the \trp\  for
rather restricted classes of \ca s.

In retrospect,
the \trp\  could be motivated as follows.
For simple \ca s,
one popular version of freeness of an action
$\af \colon G \to \Aut (A)$
is the requirement that $\af_g$ be outer
for all $g \in G \setminus \{ 1 \}.$
(This condition is called pointwise outerness in Definition~\ref{D:Out},
and it and its variants are the subject of Section~\ref{Sec:Outer}.)
Let $R$ be the hyperfinite factor of type II$_1.$
Then pointwise outer actions of finite groups
satisfy the von Neumann algebra analog of the \rp.
(Lemma 5.2.1 of~\cite{Jn} implies this statement.)
One might then ask that an action $\af \colon G \to \Aut (A)$
of a finite group $G$ on a simple separable in\fd\  \uca~$A,$
with a unique \tst~$\ta,$
have the property that,
in the weak closure of the Gelfand-Naimark-Segal representation
associated with~$\ta$ (which is isomorphic to~$R$),
the action becomes outer.
Under good conditions (see Theorem~\ref{T:OuterImpTRP} below),
this requirement is equivalent to the \trp.

The actual motivation for the \trp\  was,
however, rather different.
It was introduced for the purpose of proving
classification theorems for crossed products.
One should observe that the definition
below is, very roughly, related to the Rokhlin property
in the same way that Lin's notion of a tracially AF \ca\   %
is related to that of an AF~algebra.
(Tracially AF \ca s are as in Definition~2.1 of~\cite{Ln0}.
The condition is equivalent to tracial rank zero
as in Definition~2.1 of~\cite{Ln1};
the equivalence is Theorem~7.1(a) of~\cite{Ln1}.)

The usefulness of the \trp\  comes from the combination of
two factors:
it implies strong structural results for crossed products,
and it is common while the \rp\  is rare.
In particular, the \trp\  played a key role in the
solution of three open problems on the structure of \cp s.
See Theorems \ref{T:NCTIsAT}, \ref{T:FGSL}, and~\ref{T:HDFlip} below.
As we will see,
the actions involved do not have the \rp,
while the next weaker freeness condition,
pointwise outerness (Definition~\ref{D:Out}),
is not strong enough to make the arguments work.

\begin{dfn}\label{D:TRP}
Let $A$ be an infinite dimensional simple \uca,
and let $\af \colon G \to \Aut (A)$ be an
action of a finite group $G$ on~$A.$
We say that $\af$ has the
{\emph{tracial Rokhlin property}} if for every finite set
$F \subset A,$ every $\ep > 0,$ and every positive element $x \in A$
with $\| x \| = 1,$
there are nonzero \mops\  $e_g \in A$ for $g \in G$
such that:
\begin{enumerate}
\item\label{D:TRP:1}
$\| \af_g (e_h) - e_{g h} \| < \ep$ for all $g, h \in G.$
\item\label{D:TRP:2}
$\| e_g a - a e_g \| < \ep$ for all $g \in G$ and all $a \in F.$
\item\label{D:TRP:3}
With $e = \sum_{g \in G} e_g,$ the \pj\  $1 - e$ is \mvnt\  to a
\pj\  in the \hsa\  of $A$ generated by~$x.$
\item\label{D:TRP:4}
With $e$ as in~(\ref{D:TRP:3}), we have $\| e x e \| > 1 - \ep.$
\end{enumerate}
\end{dfn}

When $A$ is finite, the last condition is redundant.
(See Lemma~1.16 of~\cite{PhtRp1}.)
However, without it, the trivial action on ${\mathcal{O}}_2$
would have the \trp.
(It is, however, not clear that this condition is really the right extra
condition to impose.)
Without the requirement that the algebra be in\fd,
the trivial action on $\C$ would have the \trp\  %
(except for the condition~(\ref{D:TRP:4})),
for the rather silly reason that the \hsa\  in Condition~(\ref{D:TRP:3})
can't be ``small''.

\begin{pbm}\label{Pb:IzumiTRP}
Is there a reasonable formulation of the \trp\ %
in terms of central sequence algebras,
analogous to the formulation of the \rp\  %
given in Remark~\ref{R:IzumiRP}?
\end{pbm}

As for the \rp,
the \trp\  is only useful when the algebra has a sufficient
supply of \pj s.
The definition is only given for simple \ca s,
because we don't know the proper formulation
of Condition~(\ref{D:TRP:3})
without simplicity.
We discuss these issues further below.

A version of this definition for actions of $\Z$
was given in~\cite{OP1}.
The analog of Condition~(\ref{D:TRP:4}) was omitted,
and the algebra was required to be stably finite.
A slightly different version for~$\Z,$
called the tracial cyclic \rp,
appears in Definition~2.4 of~\cite{LO}.

Since we require algebras with actions with the \trp\  to
be simple, unital, and in\fd,
they can't be type~I.
Thus, one of our standard examples is irrelevant.
For product type actions of $\Zqt,$
we have:

\begin{exa}\label{E:PTypeTRP}
Let $\af \colon \Zqt \to \Aut (A)$ be as in Definition~\ref{D:PType}.
Then \tfae:
\begin{enumerate}
\item\label{E:PTypeTRP:TRP}
$\af$ has the \trp.
\item\label{E:PTypeTRP:GNS}
If $\af''$ is the action induced by $\af$
on the type~II$_1$ factor obtained
from the trace via the Gelfand-Naimark Segal construction,
then $\af''$ is outer.
\item\label{E:PTypeTRP:Prod}
For all~$N,$ we have
\[
\prod_{n = N}^{\infty} \frac{d (n) - 2 k (n)}{d (n)} = 0.
\]
\item\label{E:PTypeTRP:UniqTst}
$C^* (\Zqt, A, \af)$ has a unique \tst.
\end{enumerate}
See Proposition~2.5 of~\cite{PhtRp4},
where additional equivalent conditions are given.
\end{exa}

In particular, by comparison with Examples \ref{E:PTypeRk}
and~\ref{E:PTypeK},
the choices $d (n) = 3$ and $k (n) = 1$ for all~$n$
give an action of $\Zqt$ which has the \trp\  %
but does not have the \rp\  and is not totally K-free.

On the other hand, the \rp\  implies the \trp.
(This is trivially true for actions on all \ca s on which
the \trp\  is defined.)
Also, for actions as in Definition~\ref{D:PType},
locally discrete K-theory implies the \trp.
This last statement is misleading,
since Example~\ref{E:939} and Proposition~\ref{T:TRPImpPOut} below
show that a K-free action on a UHF~algebra need not have the \trp.

\begin{pbm}\label{P:KFImpTRP}
Let $A$ be a simple separable unital tracially AF \ca,
let $G$ be a finite group,
and let $\af \colon G \to \Aut (A)$ be totally K-free.
Does it follow that $\af$ has the \trp?
\end{pbm}

In particular,
what happens for actions on simple unital AF~algebras?

We now give the result promised in our initial discussion.

\begin{thm}[Theorem~5.5 of~\cite{ELPW}]\label{T:OuterImpTRP}
Let $A$ be a simple separable \uca\  with tracial rank zero,
and suppose that $A$ has a unique \tst~$\ta.$
Let $\pi_{\ta} \colon A \to L (H_{\ta})$
be the Gelfand-Naimark-Segal representation associated with~$\ta.$
Let $G$ be a finite group, and let $\af \colon G \to \Aut (A)$
be an action of $G$ on~$A.$
Then $\af$ has the \trp\  \ifo\  %
$\af_g''$ is an outer automorphism of
$\pi_{\ta} (A)''$ for every $g \in G \setminus \{ 1 \}.$
\end{thm}

The corresponding statement for actions of~$\Z$ is also true
(Theorem~2.18 of~\cite{OP1b}).

\begin{pbm}\label{Pb:ManyTST}
Is there a related characterization of the \trp\  for
actions on simple separable \uca s with tracial rank zero
which have more than one \tst?
\end{pbm}

Crossed products by actions with the tracial Rokhlin property
cannot be expected to be as well behaved as those by actions with
the Rokhlin property.
Indeed,
Example~\ref{E:TRP}(\ref{E:TRP:Black}) below shows that they do
not preserve AF~algebras or AI~algebras, and
Example~\ref{E:TRP}(\ref{E:TRP:Torsion}) shows that they do not
preserve AT~algebras.
The tracial Rokhlin property does imply pointwise outerness
(see Proposition~\ref{T:TRPImpPOut} below),
so that the crossed product of a simple \uca\  by an
action of a finite group with the tracial Rokhlin property
is again simple (Corollary~1.6 of~\cite{PhtRp1}).
But the tracial Rokhlin property is much stronger than
pointwise outerness.
The following theorems give classes of \ca s
which are closed under formation
of crossed products by actions of finite groups with the tracial
Rokhlin property.
Example~\ref{E:Elliott} below shows that they do not hold for
pointwise outer actions.

\begin{thm}[Theorem~2.6 of~\cite{PhtRp1}]\label{T:RokhTAF}
Let $A$ be an in\fd\  simple separable \uca\  %
with tracial rank zero.
Let $\af \colon G \to \Aut (A)$
be an action of a finite group $G$ on $A$ which has the \trp.
Then $C^* (G, A, \af)$ has tracial rank zero.
\end{thm}

It is shown in~\cite{LO} that if
$A$ is an in\fd\  simple separable \uca\  %
with tracial rank zero and $\af \colon \Z \to \Aut (A)$
is an action with the \trp\  and satisfying extra conditions
(which hold in many interesting examples),
then $C^* (\Z, A, \af)$ has tracial rank zero.

\begin{thm}[\cite{OP3}]\label{T:RokhTRn}
Let $A$ be an in\fd\  simple separable \uca\  %
with tracial rank at most~$n,$
in the sense of Definition~2.1 of~\cite{Ln1}.
Let $\af \colon G \to \Aut (A)$
be an action of a finite group $G$ on $A$ which has the \trp.
Then $C^* (G, A, \af)$ has tracial rank at most~$n.$
\end{thm}

The following theorem combines several results from~\cite{Ar1}.
The analog for actions of $\Z$ is in~\cite{OP1}.

\begin{thm}[\cite{Ar1}]\label{T:TRP_RR}
Let $A$ be a stably finite in\fd\  simple separable \uca\  %
with real rank zero and such that the order on \pj s over~$A$
is determined by traces.
(See~\cite{Ar1} for the definition of this condition.)
Let $\af \colon G \to \Aut (A)$
be an action of a finite group $G$ on $A$ which has the \trp.
Then $C^* (G, A, \af)$ has real rank zero and
the order on \pj s over~$C^* (G, A, \af)$
is determined by traces.
If moreover $A$ has stable rank one, then so does $C^* (G, A, \af).$
\end{thm}

\begin{exa}\label{E:Elliott}
Example~9 of~\cite{Ell3}
gives an example of a pointwise outer action~$\af$
(in the sense of Definition~\ref{D:Out} below) of $\Zqt$
on a simple unital AF~algebra~$A$
such that $C^* (\Zqt, A, \af)$ does not have real rank zero.
This example shows that Theorems~\ref{T:RokhTAF},
\ref{T:RokhTRn}, and~\ref{T:TRP_RR}
fail for general outer actions.
\end{exa}

\begin{exa}\label{E:RPn3}
The action $\af \colon \Zqt \to \Aut (A)$ of Example~\ref{E:RPn}
satisfies the \trp, by Proposition~4.2(3) of~\cite{PhtRp4}.
Since $A$ is a simple AF~algebra and $K_0 (C^* (\Zqt, A, \af))$
has torsion,
it follows that crossed products by actions of finite groups
with the \trp\  do not preserve any of the classes of
AF~algebras, AI~algebras, or AT~algebras.
\end{exa}

In addition, if $\af \colon G \to \Aut (A)$ has the tracial
Rokhlin property, and $A$ is simple, stably finite, and infinite
dimensional, then the restriction map from tracial states on
$C^* (G, A, \af)$ to $\af$-invariant tracial states on $A$ is
bijective.
(See Proposition~5.7 of~\cite{ELPW}.)
This is also false for general outer actions.
(See Examples \ref{E:PTypeTRP} and~\ref{E:PTypeOut}.)

The big advantage of the tracial Rokhlin property is
that it is common (at least on simple \ca s with many \pj s),
while the Rokhlin property is rare.

\begin{exa}\label{E:TRP}
The following actions have the tracial Rokhlin property but not
the Rokhlin property.
The citations are for proofs that the actions in question have the \trp;
failure of the \rp\  is discussed afterwards.
\begin{enumerate}
\item\label{E:TRP:Gauge}
The action of $\Zq{n}$ on a simple higher dimensional
noncommutative torus which multiplies one of the standard
generators by $\exp (2 \pi i / n).$
(See Proposition~2.10 of~\cite{PhtRp2}.)
\item\label{E:TRP:TorFlip}
The flip action of $\Zqt$ on a simple higher dimensional
noncommutative torus.
(See Corollary~5.11 of~\cite{ELPW}.)
\item\label{E:TRP:RotAlg}
The standard actions of $\Zq{3},$ $\Zq{4},$ and $\Zq{6}$ on an
irrational rotation algebra.
(See Corollary~5.12 of~\cite{ELPW}.)
\item\label{E:TRP:UHF}
For an arbitrary UHF~algebra, many product type actions of $\Zqt.$
(See Example~\ref{E:PTypeTRP}.)
\item\label{E:TRP:Black}
Blackadar's example~\cite{Bl0} of an action of $\Zqt$ on
$\bigotimes_{n = 1}^{\infty} M_2$
such that the crossed product is not~AF.
(See Proposition~3.4 of~\cite{PhtRp4}.)
\item\label{E:TRP:Torsion}
The actions of Example~\ref{E:RPn}.
(See Proposition~4.2 of~\cite{PhtRp4}.)
\item\label{E:TRP:AFK1}
Actions $\af \colon \Zqt \to \Aut (A)$ similar to those
of Example~\ref{E:RPn},
as in Proposition~4.6 of~\cite{PhtRp4}.
Here, $A$ is a simple AF~algebra and
$K_1 (C^* (\Zqt, A, \af)) \neq 0.$
\item\label{E:TRP:TensFlip}
The tensor flip on $A \otimes A$ for many stably finite simple
approximately divisible \ca s $A.$
(See \cite{OP3}.)
\end{enumerate}
\end{exa}

As we will see,
in many of these cases, in particular, in (\ref{E:TRP:Gauge}),
(\ref{E:TRP:TorFlip}), (\ref{E:TRP:RotAlg}), (\ref{E:TRP:UHF}) for
odd UHF~algebras,~(\ref{E:TRP:Torsion}), and many cases
of~(\ref{E:TRP:TensFlip}), there does not exist {\emph{any}}
action of the group on the \ca\  which has the Rokhlin property.

There is one obvious obstruction to the \rp.
Let $A$ be a \uca.
Let $n \in \N.$
Suppose, for simplicity,
that the ordered group $K_0 (A)$
has no nontrivial automorphisms which fix~$[1],$
and that $[1]$ is not of the form $n \et$
for any $\et \in K_0 (A).$
Then no group $G$ with $\card (G) = n$
admits any action $\af \colon G \to \Aut (A)$
with the \rp.
Simply take $\ep < 1$ in Definition~\ref{SRPDfn}
to get $\af_g (e_1)$ \mvnt\  to $e_g$ for all $g \in G,$
and use triviality of $(\af_g)_*$ to get $[\af_g (e_1)] = [e_1]$
in $K_0 (A).$
So one would get $n [e_1] = [1].$

It is now immediate that no action of
any nontrivial finite group on any irrational rotation algebra
can have the \rp.
Similarly, for any odd $m \geq 3,$
no action of $\Zqt$ on the $m^{\infty}$ UHF~algebra can have the \rp.
For the same reason, no action of $\Zqt$
on any odd Cuntz algebra or on ${\mathcal{O}}_{\infty}$
has the \rp.

In fact, existence of an action of $G$ with the \rp\  implies much
stronger restrictions on the K-theory (Theorem~3.2 of~\cite{Iz2}),
namely vanishing cohomology as $\Z [G]$-modules
for $K_* (A)$ and certain subgroups.
The following result is a special case.

\begin{prp}[\cite{Iz2}]\label{P:NoRP}
Let $n \in \N,$
let $A$ be a \uca,
and let $\af \colon \Zq{n} \to \Aut (A)$ be an action with the \rp\  %
which is trivial on $K_* (A).$
Then $K_* (A)$ is uniquely $n$-divisible.
\end{prp}

\begin{proof}
See the discussion after the proof of Theorem~3.2 of~\cite{Iz2}
\end{proof}

Proposition~\ref{P:NoRP}
rules out the actions in Parts (\ref{E:TRP:TorFlip})
and~(\ref{E:TRP:Torsion}) of Example~\ref{E:TRP}.
It also shows that if a UHF algebra admits an action
of $\Zqt$ with the \rp,
then it must tensorially absorb the $2^{\infty}$~UHF algebra.
Even on UHF algebras which satisfy this condition,
Example~\ref{E:PTypeRk} shows that
``most'' product type actions of $\Zqt$ do not have the \rp.

Theorem~\ref{T:RPKthInj} contains a K-theoretic restriction
of a different kind.

By contrast,
there is no apparent K-theoretic obstruction to the \rp\  %
for actions of~$\Z,$
and there is no action of $\Z$ which is known to have the \trp\  %
but known not to have the \rp.

The actions in parts (\ref{E:TRP:Gauge}), (\ref{E:TRP:RotAlg}),
and~(\ref{E:TRP:TorFlip}) of Example~\ref{E:TRP}
play a key role, via Theorem~\ref{T:RokhTAF}, in the proofs
of the following recent solutions to open problems on the
structure of certain crossed product \ca s.
In none of these proofs is outerness of the action sufficient
(Example~\ref{E:Elliott}
shows that crossed products by such actions do not
necessarily preserve tracial rank zero),
while on the other hand
the discussion above shows that none of the actions has the \rp.

\begin{thm}[Theorem~3.8 of~\cite{PhtRp2}]\label{T:NCTIsAT}
Let $\Th$ be a nondegenerate skew symmetric real $d \times d$ matrix,
with $d \geq 2.$
Then the noncommutative
$d$-torus $A_{\Th}$ is a simple AT~algebra with real rank zero.
\end{thm}

The relevance of actions of finite groups is that they
allow reduction of the general case
to the case in which $A_{\Th}$ can be written as an
iterated \cp\  by actions of~$\Z$
in such a way that all the intermediate crossed products are simple.
This case was solved by Kishimoto (Corollary~6.6 of~\cite{Ks4}).

The action and the subgroups which appear in the following
theorem are described,
for example, in the introduction to~\cite{ELPW}.

\begin{thm}[Theorem~0.1 of~\cite{ELPW}]\label{T:FGSL}
Let $\theta \in \R \setminus \Q.$
Let $A_{\te}$ be the irrational rotation algebra,
and let $\af \colon \SL_2 (\Z) \to \Aut (A_{\te})$
be the standard action of $\SL_2 (\Z)$ on $A_{\te}.$
Let $F$ be any of the standard finite subgroups
$\Zqt, \Zq{3}, \Zq{4}, \Zq{6} \subset \SL_2 (\Z).$
Then the crossed product $C^* (F, A_{\te}, \af |_F )$ is
an AF~algebra.
\end{thm}

(The case $F = \Zqt$ was already known~\cite{BK}.)

\begin{thm}[Theorem~0.4 of~\cite{ELPW}]\label{T:HDFlip}
Let $A_{\Theta}$ be the noncommutative
$d$-torus corresponding to a nondegenerate
skew symmetric real $d \times d$ matrix~$\Theta.$
Let
$\alpha \colon \Zqt \to \Aut (A_{\Theta})$ denote the flip action.
Then
$C^* (\Zqt, A_{\Theta}, \alpha)$ and the
fixed point algebra $A_{\Theta}^{\Zqt}$ are AF~algebras.
\end{thm}

The \trp\  has the following permanence property:

\begin{prp}[Lemma~5.6 of~\cite{ELPW}]\label{P:TRPSubgp}
Let $A$ be an infinite dimensional simple \uca,
and let $\af \colon G \to \Aut (A)$ be an
action of a finite group $G$ on $A$
which has the \trp.
Let $H$ be a subgroup of~$G.$
Then $\af |_H$ has the \trp.
\end{prp}

Permanence properties involving ideals, quotients,
and extensions don't make sense,
since the \trp\  is (so far) defined only for actions on simple \ca s.
It seems plausible that a direct limit of actions with
the \trp\  again has the \trp,
but nobody has checked this.

\begin{pbm}\label{Pb:TRPTens}
Let $G$ be a finite group,
let $A$ and $B$ be infinite dimensional simple \uca s,
let $\af \colon G \to \Aut (A)$ be an action with the \trp,
and let $\bt \colon G \to \Aut (B)$ be an arbitrary action.
Does it follow that
$\af \otimes_{\mathrm{min}} \bt \colon
   G \to \Aut (A \otimes_{\mathrm{min}} B)$
has the \trp?
(We use the minimal tensor product to ensure simplicity.)
\end{pbm}

Lemma~3.9 of~\cite{PhtRp1} is the very
special case $B = M_n$ and $\bt$ is inner.
Proposition~4.3 of~\cite{LO} gives a related result for
actions of $\Z$ which have the tracial cyclic \rp,
Definition~2.4 of~\cite{LO}.
The assumptions are
that $A$ is simple, unital, and has tracial rank zero,
that $B$ is simple, unital, and has tracial rank at most one,
that $\af \in \Aut (A)$ has the tracial cyclic \rp,
and that $\bt \in \Aut (B)$ is arbitrary.
The conclusion is that
$\af \otimes_{\mathrm{min}} \bt \in \Aut (A \otimes_{\mathrm{min}} B)$
has the tracial cyclic \rp.
The same proof gives the following result,
pointed out to us by Hiroyuki Osaka:

\begin{prp}[Osaka]\label{P:TRPTensPr}
Let $G$ be a finite group,
let $A$ and $B$ be infinite dimensional simple \uca s,
let $\af \colon G \to \Aut (A)$ be an action with the \trp,
and let $\bt \colon G \to \Aut (B)$ be an arbitrary action.
Suppose $A$ has tracial rank zero
and $B$ has tracial rank at most one.
Then
$\af \otimes_{\mathrm{min}} \bt \colon
   G \to \Aut (A \otimes_{\mathrm{min}} B)$
has the \trp.
\end{prp}

The key point is that the condition on $1 - \sum_{g \in G} e_g$
in Definition~\ref{D:TRP}(\ref{D:TRP:3})
can be verified by using the values of tracial states
on this element.
A general positive solution to Problem~\ref{Pb:TRPTens}
requires relating \hsa s in
$A \otimes_{\mathrm{min}} B$ to \hsa s in~$A,$
which might be difficult.

The \trp, as given in Definition~\ref{D:TRP},
suffers from three major defects:
the algebra must be unital,
it must have many \pj s,
and it must be simple.

Presumably the nonunital simple case can be handled by
something like Definition~\ref{D:MRP}.
However, the correct analog of Condition~(\ref{D:TRP:3})
of Definition~\ref{D:TRP} is not clear.

Archey~\cite{Ar2} has made progress toward handling
the simple unital case with few \pj s.
We refer to~\cite{Ar2} for unexplained terminology
in the following.

\begin{dfn}[\cite{Ar2}]\label{D:NoPjTRP}
Let $A$ be an infinite dimensional stably finite simple \uca,
and let $\af \colon G \to \Aut (A)$ be an
action of a finite group $G$ on $A.$
We say that $\af$ has the
{\emph{projection free tracial Rokhlin property}}
if for every finite set
$F \subset A,$ every $\ep > 0,$ and every positive element $x \in A$
with $\| x \| = 1,$
there are mutually orthogonal positive elements $a_g \in A$
for $g \in G$ with $\| a_g \| = 1$ for all $g \in G,$
such that:
\begin{enumerate}
\item\label{D:NoPjTRP:1}
$\| \af_g (a_h) - a_{g h} \| < \ep$ for all $g, h \in G.$
\item\label{D:NoPjTRP:2}
$\| a_g c - c a_g \| < \ep$ for all $g \in G$ and all $c \in F.$
\item\label{D:NoPjTRP:3}
With $a = \sum_{g \in G} a_g,$ we have $\ta (1 - a) < \ep$
for every \tst\  $\ta$ on~$A.$
\item\label{D:NoPjTRP:4}
With $a = \sum_{g \in G} a_g,$ the element $1 - a$
is Cuntz subequivalent to an element
of the \hsa\  of $A$ generated by~$x.$
\end{enumerate}
\end{dfn}

For example, let $Z$ be the Jiang-Su algebra~\cite{JS}.
Then $Z \otimes Z$ has no nontrivial \pj s.
Archey shows~\cite{Ar2} that the tensor flip on $Z \otimes Z$
generates an action of $\Zqt$ with the projection free \trp.
The other hypotheses of the following theorem are also satisfied.
Again, see~\cite{Ar2} for unexplained terminology.

\begin{thm}[\cite{Ar2}]\label{T:NoPj}
Let $A$ be an infinite dimensional stably finite simple \uca\  %
with stable rank one and with strict comparison of positive elements.
Further assume that every $2$-quasitrace on $A$ is a trace,
and that $A$ has only finitely many extreme \tst s.
Let $\af \colon G \to \Aut (A)$ be an
action of a finite group $G$ on $A$
which has the projection free \trp.
Then $C^* (G, A, \af)$ has stable rank one.
\end{thm}

One possible next step is to ask whether there is an
analog of Theorem~\ref{T:OuterImpTRP} using the
projection free \trp,
say for actions on simple separable unital nuclear $Z$-stable \ca s
with a unique \tst.
For that matter,
one might try using outerness in the
Gelfand-Naimark-Segal representation associated with a \tst\  as
a hypothesis for theorems on preservation of structure
in crossed products.

Finding the right definition for nonsimple \ca s seems to
be the most difficult problem.
One guide is that a free action on a compact metric space
should presumably have the \trp.
There is work in progress for actions of $\Z$
on quite special nonsimple \ca s.

\section{Pointwise outerness}\label{Sec:Outer}

\indent
Pointwise outerness is easy to define and,
at least for discrete groups acting on simple \ca s,
has useful consequences.
This and related conditions have mostly been used to
prove simplicity of crossed products $C^* (G, A, \af)$
when $G$ is discrete
and $A$ has no nontrivial $G$-invariant ideals,
or, more generally, that every ideal in $C^* (G, A, \af)$
is the crossed product by an invariant ideal of~$A.$
There are also theorems on preservation of structure,
for example for pure infiniteness and Property~(SP).

Like K-theoretic freeness,
pointwise outerness does not pass to invariant ideals
or their quotients.
A useful condition for actions on nonsimple algebras
must therefore be stronger.
A number of strengthenings have been used.
Recently introduced conditions include topological freeness~\cite{AS},
essential freeness of the action on the space of unitary equivalence
classes of irreducible representations~\cite{Sr},
and the Rokhlin* property~\cite{Sr}.
In Definition~\ref{D:SPOut} below,
we give another possible strengthening:
requiring pointwise outerness for all actions of subgroups
on invariant subquotients of the algebra.
But we do not know how useful this condition is.

\begin{dfn}\label{D:Out}
An action $\af \colon G \to \Aut (A)$ is said to be
{\emph{pointwise outer}}
if, for $g \in G \setminus \{ 1 \},$
the \am\  $\af_g$ is outer, that is,
not of the form $a \mapsto \Ad (u) (a) = u a u^*$
for some unitary $u$ in the multiplier algebra $M (A)$ of $A.$
\end{dfn}

Such actions are often just called outer.

An inner action $\af \colon G \to \Aut (A)$
is one for which there exists
a \hm\  $g \mapsto u_g,$ from $G$ to the unitary group 
of $M (A),$
such that $\af_g = \Ad (u_g)$ for all $g \in G.$
(If $G$ is not discrete, one should impose
a suitable continuity condition.)
There exist (see Example~\ref{E:4Gp} below)
actions of finite groups which are pointwise inner
(so that each $\af_g$ has the form $\Ad (u_g)$)
but not inner
(it is not possible to choose $g \mapsto u_g$ to be a group \hm).

In the literature,
a single automorphism is often called {\emph{aperiodic}}
if it generates a pointwise outer action of~$\Z.$

\begin{exa}\label{E:PTypeOut}
Let $\af \colon \Zqt \to \Aut (A)$ be as in Definition~\ref{D:PType}.
Then \tfae:
\begin{enumerate}
\item\label{E:PTypeOut:Out}
$\af$ is pointwise outer.
\item\label{E:PTypeOut:Smp}
$C^* (\Zqt, A, \af)$ is simple.
\item\label{E:PTypeOut:NTr}
For infinitely many $n,$ we have $k (n) \neq 0.$
\end{enumerate}
See Proposition~2.6 of~\cite{PhtRp4},
where additional equivalent conditions are given.
\end{exa}

By comparison with Example~\ref{E:PTypeTRP},
the choices $d (n) = 2^n$ and $k (n) = 1$ for all~$n$
give a pointwise outer action of $\Zqt$ which does not have the \trp.
Moreover, $\af$ becomes inner on the double commutant
of the Gelfand-Naimark-Segal representation from
the unique \tst\  on~$A.$
However, for examples of this type,
it follows that the \trp\  implies pointwise outerness.
This is true in general.

\begin{prp}[Lemma~1.5 of~\cite{PhtRp1}]\label{T:TRPImpPOut}
Let $A$ be an infinite dimensional simple \uca,
and let $\af \colon G \to \Aut (A)$ be an
action of a finite group $G$ on~$A.$
If $\af$ has the \trp, then $\af$ is pointwise outer.
\end{prp}

For type~I \ca s,
see Theorem~\ref{T:StOuterType1}
and the discussion before Definition~\ref{D:SPOut} below.

The main application so far of pointwise outerness has been
to proofs of simplicity of reduced crossed products.
This application is valid for general discrete groups.
The next theorem is due to Kishimoto.
The expression ${\widetilde{\Gm}} (\bt)$ is as in
Definition~\ref{D:StrongConnes} below,
and requiring that it be nontrivial is a
strong outerness condition.
The second result is a corollary of the first.

For the statement of this and several later results,
the following definition is convenient.
It generalizes a standard definition for actions on topological spaces.

\begin{dfn}\label{D:Min}
An action $\af \colon G \to \Aut (A)$
is {\emph{minimal}} if $A$ has no $\af$-invariant ideals
other than $\{ 0 \}$ and~$A.$
\end{dfn}

\begin{thm}[Theorem~3.1 of~\cite{Ks1}]\label{T:Ks1}
Let $\af \colon G \to \Aut (A)$ be a minimal action of a
discrete group $G$ on a \ca~$A.$
Suppose that ${\widetilde{\Gm}} (\af_g)$
(with $\af_g$ being regarded as an action of $\Z$)
is nontrivial for every $g \in G \setminus \{ 1 \}.$
Then the reduced \cp\  $C^*_{\mathrm{r}} (G, A, \af)$ is simple.
\end{thm}

\begin{thm}[Part of Theorem~3.1 of~\cite{Ks1}]\label{T:Ks}
Let $\af \colon G \to \Aut (A)$ be an action of a
discrete group $G$ on a simple \ca~$A.$
Suppose that $\af$ is pointwise outer.
Then $C^*_{\mathrm{r}} (G, A, \af)$ is simple.
\end{thm}

We note the following generalization,
which is the corollary after Theorem~1 in~\cite{AS}.
In this theorem and the next,
${\widehat{A}}$ is the space of unitary
equivalence classes of irreducible representations of the \ca~$A,$
with the hull-kernel topology.

\begin{thm}[\cite{AS}]\label{T:FullAS}
Let $\af \colon G \to \Aut (A)$ be a minimal action of a
discrete group $G$ on a \ca~$A.$
Suppose that $\af$ is topologically free,
that is, for every finite set $F \subset G \setminus \{ 1 \},$
the set
\[
\big\{ x \in {\widehat{A}}
 \colon {\mbox{$g x \neq x$ for all $g \in F$}} \big\}
\]
is dense in ${\widehat{A}}.$
Then $C^*_{\mathrm{r}} (G, A, \af)$ is simple.
\end{thm}

In the following result of Sierakowski,
minimality is not required,
and the conclusion is accordingly that all ideals in the \cp\  %
are \cp s of ideals in the original algebra.
The proof involves applying Theorem~\ref{T:FullAS}
to invariant quotients.
Exact actions are as in Definition~1.2 of~\cite{Sr}.
In particular, every action of an exact group is exact.
Example~\ref{E:WeakPtOuter} below shows that topological freeness
does not suffice in this theorem.

\begin{thm}[Theorem~1.16 of~\cite{Sr}]\label{T:Sr}
Let $\af \colon G \to \Aut (A)$ be an exact action of a
discrete group $G$ on a \ca~$A.$
Suppose that the action of $G$ on ${\widehat{A}}$ is essentially free,
that is,
for every $G$-invariant closed subset $X \S {\widehat{A}},$
the subset
\[
\big\{ x \in X \colon
 {\mbox{$g x \neq x$ for all $g \in G \setminus \{ 1 \}$}} \big\}
\]
is dense in~$X.$
Then every ideal $J \subset C^* (G, A, \af)$ has the form
$C^* (G, I, \af_{(\cdot)} |_I)$
for some $G$-invariant ideal $I \subset A.$
\end{thm}

In Theorem~2.5 of~\cite{Sr},
the same conclusion is obtained using the Rokhlin* property
(Definition~2.1 of~\cite{Sr})
in place of essential freeness.
The Rokhlin* property is a weaker hypothesis,
by Theorem~2.10 of~\cite{Sr}.
For finite groups, the \rp\  implies the Rokhlin* property,
but the Rokhlin* property is much weaker than the \rp,
involving \pj s in the second dual of quotients of the algebra.

It is built into the definitions of both essential freeness
of the action on ${\widehat{A}}$
and the Rokhlin* property that these properties pass to quotients
by invariant ideals.

Pointwise outerness of an action on a simple \ca\  %
also implies that the crossed product preserves
pure infiniteness
and Property~(SP)
(every nonzero hereditary subalgebra contains a nonzero \pj).
The following two results
are special case of Corollaries 4.4 and~4.3 of~\cite{JO}.
(For the first, one also needs Theorem~\ref{T:Ks1}.)

\begin{thm}[\cite{JO}]\label{T:PICrPrd}
Let $\af \colon G \to \Aut (A)$ be a pointwise outer action of a
discrete group $G$ on a unital purely infinite simple \ca~$A.$
Then $C^* (G, A, \af)$ is purely infinite simple.
\end{thm}

\begin{thm}[\cite{JO}]\label{T:SPCrPrd}
Let $\af \colon G \to \Aut (A)$ be a pointwise outer action of a
discrete group $G$ on a unital simple \ca~$A$ with Property~(SP).
Then $C^* (G, A, \af)$ has Property~(SP).
\end{thm}

In the general statement of Corollary~4.4 of~\cite{JO},
it is allowed that
\[
N = \{ g \in G \colon {\mbox{$\af_g$ is inner}} \}
\]
be finite instead of necessarily trivial.
The conclusion
is then that $C^* (G, A, \af)$ is purely infinite
but not necessarily simple.
In Corollary~4.3 of~\cite{JO},
dealing with Property~(SP),
if $G$ is finite then $\af$ can be arbitrary.

Definition~\ref{D:Out}
is suitable only for actions on simple \ca s,
since one can always take the direct sum of an outer action
on one \ca\  and the trivial action on another.
Although, to our knowledge, nothing has been done with it,
the most obvious way to rule out such actions seems to be the following
definition.
It is motivated by Theorem~\ref{CT:Outer} and adapted from~\cite{PP}.

\begin{dfn}\label{D:SPOut}
An action $\af \colon G \to \Aut (A)$ is said to be
{\emph{strongly pointwise outer}}
if, for every $g \in G \setminus \{ 1 \}$
and any two $\af_g$-invariant ideals $I \subset J \subset A$
with $I \neq J,$
the automorphism of $J / I$ induced by $\af_g$ is outer.
\end{dfn}

\begin{thm}\label{T:StOuterType1}
Let $\af \colon G \to \Aut (A)$
be an action of a compact metrizable group~$G$
on a type~I \ca~$A.$
Then $\af$ is strongly pointwise outer
\ifo\  %
$\af$ induces a free action on $\Prim (A).$
\end{thm}

\begin{proof}
Suppose the action of $G$ on $\Prim (A)$ is not free.
We prove that $\af$ is not strongly pointwise outer.
Corollary~8.1.2 of~\cite{Ph1} provides a
$G$-invariant composition series $(I_{\ld})_{\ld \leq \kp},$
for some ordinal~$\kp,$
such that each composition factor $I_{\ld + 1} / I_{\ld}$
has Hausdorff primitive ideal space.
Choose $g \in G \setminus \{ 1 \},$
$\ld < \kp,$ and $P \in \Prim (I_{\ld + 1} / I_{\ld})$
such that $g P = P.$
Set $B = I_{\ld + 1} / I_{\ld}.$
Then $\af_g$ descends
to an automorphism of $B / P.$
Since $B$ has type~I and $B / P$ is simple,
there is a Hilbert space~$H$ such that $B / P \cong K (H).$
Because all automorphisms of $K (H)$ are inner,
we have contradicted strong pointwise outerness.

Now suppose the action on $\Prim (A)$ is free.
Let $g \in G \setminus \{ 1 \}$
and let $I \subset J \subset A$
be $\af_g$-invariant ideals 
with $I \neq J.$
Then $\Prim (J / I)$ is a nonempty $g$-invariant subset of $\Prim (A),$
and the automorphism of $J / I$ induced by $\af_g$ is therefore
nontrivial on $\Prim (J / I).$
Thus, this automorphism can't be inner.
\end{proof}

For any \ca~$A,$
free action of the group on $\Prim (A)$
clearly implies the conditions used in
Theorems \ref{T:FullAS} and~\ref{T:Sr}.
Hence, by Theorem~2.10 of~\cite{Sr},
this condition implies the Rokhlin* property.

The following example shows that
in Definition~\ref{D:SPOut} it is not enough to assume that
the action on every $\af_g$-invariant ideal is outer,
or that
the action on the quotient by every $\af_g$-invariant ideal is outer,
or even both.
The action fails to have the Rokhlin* property
and is not essentially free on ${\widehat{A}}.$
However, it is topologically free in the sense
used in Theorem~\ref{T:FullAS}.
Topological freeness is therefore not enough to obtain the
conclusion of Theorem~\ref{T:Sr}.

\begin{exa}\label{E:WeakPtOuter}
We construct an action $\af \colon G \to \Aut (A),$
with $G = \Zqt$ and
in which $A$ is a separable type~I \ca,
with the following properties:
\begin{enumerate}
\item\label{E:WeakPtOuter:OQt}
For every $g \in G \setminus \{ 1 \}$
and any $\af_g$-invariant ideal $I \subset A$
with $I \neq A,$
the automorphism of $A / I$ induced by $\af_g$ is outer.
\item\label{E:WeakPtOuter:OId}
For every $g \in G \setminus \{ 1 \}$
and any $\af_g$-invariant ideal $I \subset A$
with $I \neq \{ 0 \},$
the automorphism of $I$ induced by $\af_g$ is outer.
\item\label{E:WeakPtOuter:NotSPO}
$\af$ is not strongly pointwise outer.
\item\label{E:WeakPtOuter:NF}
The action of $G$ on $\Prim (A)$ is not free.
\item\label{E:WeakPtOuter:TopF}
$\af$ is topologically free in the sense used in Theorem~\ref{T:FullAS}.
\item\label{E:WeakPtOuter:NEss}
The action of $G$ on ${\widehat{A}}$ is not essentially free
in the sense used in Theorem~\ref{T:Sr}.
\item\label{E:WeakPtOuter:NRSt}
$\af$ does not have the Rokhlin* property.
\item\label{E:WeakPtOuter:ExI}
There is an ideal in $C^* (G, A, \af)$ not of the form
$C^* (G, I, \af_{(\cdot)} |_I)$
for any $G$-invariant ideal $I \subset A.$
\end{enumerate}

Let $K = K (H)$ be the algebra of compact operators on a separable
infinite dimensional Hilbert space~$H.$
Let $e_1, e_2 \in L (H)$ be two infinite rank \pj s
such that $e_1 + e_2 = 1,$
and let $u \in L (H)$ be a unitary such that $u e_1 u^* = e_2$
and $u^2 = 1.$
Set $B_0 = K + \C e_1 + \C e_2 \subset L (H),$
and set $B = (K \otimes B_0)^+,$
the unitization of $K \otimes B_0.$
Let $\bt_0 \colon \Zqt \to \Aut (B_0)$ be the action
generated by $\Ad (u).$
(Note that $u B_0 u^* = B_0$ even though $u \not\in B_0.$)
Then let $\bt \colon \Zqt \to \Aut (B)$ be the action
$(\id_K \otimes \bt_0)^+.$
Finally set $A = C^* (\Zqt, B, \bt),$
and let $\af \colon \Zqt \to \Aut (A)$ be the dual action
$\af = {\widehat{\bt}}.$

It follows from Takai duality (7.9.3 of~\cite{Pd1})
and the fact that crossed products preserve exact sequences
(Lemma~2.8.2 of~\cite{Ph1})
that the map $I \mapsto C^* (\Zqt, I, \bt |_I)$
defines a bijection
between the $\bt$-invariant ideals of~$B$
and the $\af$-invariant ideals of~$A.$
The only nontrivial $\bt$-invariant ideals of~$B$ are
\[
I = K \otimes K
\andeqn
J = K \otimes B_0.
\]
Accordingly,
the only nontrivial $\af$-invariant ideals of~$A$ are
\[
L = C^* \big( \Zqt, \, K \otimes K,
             \, \bt_{( \cdot )} |_{K \otimes K} \big)
\andeqn
M = C^* \big( \Zqt, \, K \otimes B_0,
             \, \bt_{( \cdot )} |_{K \otimes B_0} \big).
\]

The action on $K \otimes K$ is inner,
so $L \cong (K \otimes K) \oplus (K \otimes K),$
and $\af$ exchanges the two summands.
Since every nonzero invariant ideal in $A$ contains~$L,$
it follows that the action on every such ideal is outer.
This proves~(\ref{E:WeakPtOuter:OId}).

The action on $B / J \cong \C$ is trivial,
so $A / M \cong \C \oplus \C,$
and $\af$ exchanges the two summands.
Since every invariant ideal in $A,$ other than $A$ itself,
is contained in~$M,$
it follows that the action on the quotient by every such ideal is outer.
This proves~(\ref{E:WeakPtOuter:OQt}).

However, the induced action on $\Prim (A)$ is not free,
and the subquotient
$M / L \cong C^* \big( \Zqt, \, K \oplus K, \, {\overline{\bt}} \big)$
(where ${\overline{\bt}}$ exchanges the two summands)
is a nonzero invariant subquotient of $A$
isomorphic to $K \otimes M_2$
on which the action $\af$ is inner.
Thus, we have (\ref{E:WeakPtOuter:NotSPO}) and~(\ref{E:WeakPtOuter:NF}).

Since $B$ has ideals which are not $\bt$-invariant
(such as $K \otimes (K + \C e_1)$),
Takai duality implies that $C^* (\Zqt, A, \af)$
has ideals which are not ${\widehat{\af}}$-invariant.
Such ideals are not crossed products of invariant ideals in~$A.$
This proves~(\ref{E:WeakPtOuter:ExI}).

We prove~(\ref{E:WeakPtOuter:TopF}).
Since $A$ has type~I,
we have ${\widehat{A}} = \Prim (A).$
We calculate $\Prim (A).$
As we saw above,
we can write $L = L_1 \oplus L_2$
for ideals $L_1, L_2 \S L$ which are exchanged by~$\af.$
Both $L_1$ and $L_2$ are easily seen to be primitive.
The ideal $L$ is itself primitive,
since (as we saw above) $M / L \cong K \otimes M_2,$
and is a fixed point.
The isomorphism $A / M \cong \C \oplus \C$
shows that $M$ is not primitive, but gives two
more primitive ideals $P_1, P_2 \S A,$
such that $A / P_1 \cong A / P_2 \cong \C.$
These are exchanged by~$\af.$
Thus
\[
\Prim (A) = \{ L_1, L_2, L, P_1, P_2 \}.
\]
The closed sets are
\[
\varnothing,
\,\,\,\,\,\,
\{ P_1 \},
\,\,\,\,\,\,
\{ P_2 \},
\,\,\,\,\,\,
\{ P_1, P_2 \},
\,\,\,\,\,\,
\{ L, P_1, P_2 \},
\]
\[
\{ L_1, L, P_1, P_2 \},
\,\,\,\,\,\,
\{ L_2, L, P_1, P_2 \},
\,\,\,\,\,\,
{\mbox{and}}
\,\,\,\,\,\,
\Prim (A).
\]
For the nontrivial group element~$g,$
we thus have
\[
\big\{ x \in {\widehat{A}} \colon g x \neq x \big\}
 = \{ L_1, L_2, P_1, P_2 \},
\]
which is dense in $\Prim (A).$

The action on $\Prim (A)$ is not essentially free,
because $\{ L, P_1, P_2 \}$ is a closed set in which the
points not fixed by~$g$ are not dense.
This is~(\ref{E:WeakPtOuter:NEss}).
The statement~(\ref{E:WeakPtOuter:NRSt})
follows from~(\ref{E:WeakPtOuter:ExI}) and Theorem~2.5 of~\cite{Sr}.
\end{exa}

The following example shows that it is not enough to consider
only subquotients invariant under the entire group,
even when the algebra is commutative.
In this example, the action is not topologically free.

\begin{exa}\label{E:S3}
Let $G$ be a finite group, and let $H \subset G$ be a
nontrivial subgroup such that $\bigcap_{g \in G} g H g^{-1} = \{ 1 \}.$
For example, take $G$ to be the symmetric group $S_3$
and take $H$ to be one of its two element subgroups.
Let $G$ act on $X = G / H$ by translation,
and let $\af$ be the corresponding action on $A = C (G / H).$
In the example using an order two subgroup of~$S_3,$
the action of $S_3$ on $X$ is just the usual action of $S_3$
on a three element set by permutations.

The stabilizer of $g H \in X$ is $g H g^{-1}.$
Since $\bigcap_{g \in G} g H g^{-1} = \{ 1 \},$
every element of $G \setminus \{ 1 \}$ acts nontrivially on~$X,$
so that $\af$ is pointwise outer.
Since $\af$ is minimal,
$\af$ is pointwise outer on $J / I$ for every pair
of $G$-invariant ideals $I \subset J$ with $I \neq J.$
However, one easily checks that $\af$ is not strongly pointwise outer.
(This also follows from Theorem~\ref{T:StOuterType1},
or from Theorem~\ref{CT:Outer}.)

Corollary~2.10 of~\cite{Gr2} implies that
$C^* (G, \, G / H) \cong K (L^2 (G / H)) \otimes C^* (H).$
Even though the action is minimal,
this \cp\  is not simple.
\end{exa}

The analog of Theorem~\ref{T:Ks} would be a positive solution
to the following problem.

\begin{pbm}\label{P:I-24a}
Let $\af \colon G \to \Aut (A)$
be a strongly pointwise outer action of a countable discrete group.
Does it follow that
every ideal $J \subset C^*_{\mathrm{r}} (G, A, \af)$ has the form
$C^*_{\mathrm{r}} (G, I, \af_{(\cdot)} |_I)$
for some $G$-invariant ideal $I \subset A$?
\end{pbm}

Note that the desired conclusion fails in Examples~\ref{E:WeakPtOuter}
and~\ref{E:S3}.

As far as we can tell,
this problem is still open,
even when $G = \Z$ and $\af$ is minimal.
(If $G = \Z$ and one assumes there are no nontrivial
$\af_g$-invariant ideals for all $g \in G \setminus \{ 0 \},$
then the desired conclusion holds.
One substitutes Theorem~2.1 of~\cite{Ks5}
for Lemma~1.1 of~\cite{Ks1} in the reasoning leading to the
proof of Theorem~3.1 of~\cite{Ks1}.)

The converse is false.
Theorem~\ref{T:FullAS} covers some actions
on algebras of the form $C (X)$ by
(necessarily nonabelian) groups which are not free,
and the crossed product in Example~\ref{E:4Gp} below
is simple even though all the automorphisms are inner.

The extra hypotheses in Theorems~\ref{T:Ks1}, \ref{T:FullAS},
and~\ref{T:Sr},
as well as the Rokhlin* property of~\cite{Sr}
and the notion of proper outerness used by Elliott~\cite{El0},
can be thought of as ways of getting around the failure of
outerness to have good permanence properties
(as shown by Examples~\ref{E:WeakPtOuter} and~\ref{E:S3}).
The question is whether these difficulties are solved
by asking for strong pointwise outerness.

Like the \trp,
the \rp\  implies strong pointwise outerness:

\begin{prp}\label{T:RPImpSOut}
Let $A$ be a \uca, let $G$ be a finite group,
and let $\af \colon G \to \Aut (A)$ have the \rp.
Then $\af$ is strongly pointwise outer.
\end{prp}

\begin{proof}
Let $g \in G \setminus \{ 1 \}.$
Using Proposition~\ref{T:RPerm}(\ref{T:RPerm:Sbgp}),
we may assume that $G$ is generated by $g.$
Let $I \subset J \subset A$ be $\af_g$-invariant ideals with $I \neq J.$
Using Proposition~\ref{T:RPerm}(\ref{T:RPerm:Ideal}),
we may assume that $I = 0.$

Suppose $\af_g |_J$ is inner, and let $u \in M (J)$
be a unitary such that $\af_g (a) = u a u^*$ for all $a \in J.$
Choose a $G$-invariant element $a \in J$ such that
$\| a \| = 1.$
Set $n = \card (G)$ and $\ep = 1 / (15 n^2).$
Let $(e_g)_{g \in G}$ be a
family of Rokhlin projections in $A$
for $\af,$ $\{ a, a u^* \},$ and~$\ep.$
(Note that $a u^* \in J \subset A.$)
Write $a = \sum_{h \in G} e_h a.$
Since $\| \af_h (e_1) - e_h \| < \ep$ and $\af_h (a) = a$
for $h \in G,$
we get $\| \af_h (e_1 a) - e_h a \| < \ep$ for all $h \in G.$
Thus
\[
1 = \| a \|
  \leq \sum_{h \in G} \| e_h a \|
  < \sum_{h \in G} ( \| \af_h (e_1 a) \| + \ep )
  = n ( \| e_1 a \| + \ep ).
\]
So $\| e_1 a \| > \frac{1}{n} - \ep.$

Using at the first step $\af_g (e_1 a) = u e_1 a u^*$
(since $e_1 a \in J$) and also $u a u^* = a$ (since $\af_g (a) = a$),
we now get
\begin{align*}
\| e_g a - e_1 a \|
& \leq \| e_g a - \af_g (e_1 a) \|
          + \| u e_1 a u^* - u a u^* e_1 \| + \| a e_1 - e_1 a \|  \\
& \leq \| e_g - \af_g (e_1) \|
          + \| e_1 (a u^*) - ( a u^*) e_1 \| + \| a e_1 - e_1 a \|
  < 3 \ep.
\end{align*}
Therefore, using $\ep \leq \frac{1}{2 n},$
\begin{align*}
\frac{1}{4 n^2}
& \leq \big( \tfrac{1}{n} - \ep \big)^2
  < \| e_1 a \|^2
  = \| a^* e_1 e_1 a \|   \\
& \leq \| a^* e_1 e_g a \| + \| a^* e_1 \| \cdot \| e_g a - e_1 a \|
  < \| a^* e_1 e_g a \| + 3 \ep.
\end{align*}
Since $e_1 e_g = 0,$
it follows that
\[
\frac{1}{4 n^2} < 3 \ep = \frac{1}{5 n^2}.
\]
This contradiction shows that $\af_g |_J$ is in fact not inner.
\end{proof}

\begin{cor}\label{T:RPImpFree}
Let $A$ be a unital type~I \ca, let $G$ be a finite group,
and let $\af \colon G \to \Aut (A)$ have the \rp.
Then the induced action of $G$ on $\Prim (A)$ is free.
\end{cor}

\begin{proof}
Apply Theorem~\ref{T:StOuterType1}
and Proposition~\ref{T:RPImpSOut}.
\end{proof}

We give a short discussion of the permanence properties
of strong pointwise outerness.
That strong pointwise outerness passes to ideals
and quotients is built into the definition;
Example~\ref{E:WeakPtOuter} shows it was necessary to do so.
That strong pointwise outerness passes to actions of subgroups
is also built into the definition;
Example~\ref{E:S3} shows that this was also necessary.
We have not investigated whether the direct limit of
strongly pointwise outer actions is strongly pointwise outer.
We have also not investigated whether
a tensor product of a strongly pointwise outer action with
another action is again strongly pointwise outer,
although a very special (and easy to prove) case of this
appears in Lemma~\ref{L:ZpTens} below.

We have said little about purely infinite simple \ca s in this survey.
There is some evidence that some of our freeness conditions collapse
for such algebras, or at least for Kirchberg algebras
(separable nuclear unital purely infinite simple \ca s).
Theorem~1 of~\cite{Nk1} shows that for actions of~$\Z$
on unital Kirchberg algebras,
pointwise outerness implies the \rp.
The examples below show that
nothing this strong can be true for actions of finite groups,
but it is possible that,
say, pointwise outerness implies the tracial \rp.
See Problem~\ref{Pb:OutImpTRP} below.

We need several lemmas for the proofs of properties of
some of our examples.

\begin{lem}\label{L:Inn}
Let $A$ be a \uca\  with trivial center,
let $n \in \N,$ and
let $\af \colon \Zq{n} \to \Aut (A)$ be an action such that
each automorphism $\af_g,$ for $g \in \Zq{n},$
is inner.
Then $\af$ is an inner action,
that is, there is a \hm\  $g \mapsto z_g$ from $\Zq{n}$ to the
unitary group $U (A)$ of $A$
such that $\af_g = \Ad (z_g)$
for $g \in \Zq{n}.$
\end{lem}

\begin{proof}
Write the elements of $\Zq{n}$ as $0, 1, \ldots, n - 1.$
Since $\af_1$ is inner,
there is a unitary $v \in A$ such that $\af_1 (a) = v a v^*$
for all $a \in A.$
In particular,
$\af_1 (v) = v v v^* = v.$
Therefore $\af_1^n (a) = v^n a v^{-n}$ for all $a \in A.$
Since $\af_1^n = \id_A$ and the center of~$A$ is trivial,
it follows that there is $\ld \in \C$ such that
$v^n = \ld \cdot 1.$
Choose $\gm \in \C$ such that $\gm^n = \ld.$
Then $z = \gm^{-1} v$ is a unitary in~$A$
with $z^n = 1$
such that $\af^k (a) = z^k a z^{-k}$
for $k = 0, 1, \ldots, n - 1$ and all $a \in A.$
\end{proof}

We don't know any counterexamples without the hypothesis
that $A$ have trivial center,
although we assume they exist.

\begin{lem}\label{L:Zp}
Let $A$ be a simple \uca, let $p$ be a prime,
and let $\af \colon \Zq{p} \to \Aut (A)$ be an action of
$\Zq{p}$ on~$A.$
If $C^* (\Zq{p}, A, \af)$ is simple, then $\af$ is pointwise outer.
\end{lem}

\begin{proof}
If $\af$ is not pointwise outer,
then, since $p$ is prime, $\af_g$ is inner for all $g \in G.$
So Lemma~\ref{L:Inn} implies that $\af$ is inner.
In this case, the \cp\  is isomorphic to $C^* (\Zq{p}) \otimes A,$
which is not simple.
\end{proof}

Example~\ref{E:939}
shows that this lemma fails if $p$ is not a prime.
The action there is not pointwise outer but the \cp\  is simple.

\begin{lem}\label{L:ZpTens}
Let $A$ be a simple \uca, let $p$ be a prime,
and let $\af \colon \Zq{p} \to \Aut (A)$ be a pointwise outer action of
$\Zq{p}$ on~$A.$
Let $B$ be any simple \uca.
Then the action
$g \mapsto (\af \otimes_{\mathrm{min}} \id_B)_g
    = \af_g \otimes_{\mathrm{min}} \id_B$
of $\Zq{p}$ on $A \otimes_{\mathrm{min}} B$
is pointwise outer.
\end{lem}

\begin{proof}
By Theorem~\ref{T:Ks}, the algebra
$C^* (\Zq{p}, A, \af)$ is simple.
Therefore so is
\[
C^* (\Zq{p}, A, \af) \otimes_{\mathrm{min}} B
\cong C^* (\Zq{p}, \, A \otimes_{\mathrm{min}} B,
      \, \af \otimes_{\mathrm{min}} \id_B).
\]
So Lemma~\ref{L:Zp} implies that
$\af \otimes_{\mathrm{min}} \id_B$ is pointwise outer.
\end{proof}

The following three examples show that,
for actions of $\Zqt$ on unital Kirchberg algebras
which satisfy the Universal Coefficient Theorem,
pointwise outerness, K-freeness, and the \rp\  %
are all distinct,
even in situations in which the K-theory is sufficiently
nontrivial that K-freeness should be useful.

\begin{exa}\label{E:PIUHF}
In Definition~\ref{D:PType},
take $d (m) = 3$ and $k (m) = 1$ for all~$m.$
Let $\af \colon \Zqt \to \Aut (A)$ be the resulting action.
Then $B = {\mathcal{O}}_{\infty} \otimes A$
is a unital Kirchberg algebra
which satisfies the Universal Coefficient Theorem.
The action $\af$ is pointwise outer by Example~\ref{E:PTypeOut}.
So $\bt = \id_{{\mathcal{O}}_{\infty}} \otimes \af$
is pointwise outer by Lemma~\ref{L:ZpTens}.
The K\"{u}nneth formula (Theorem~4.1 of~\cite{Sc2})
implies that $a \mapsto 1 \otimes a$
induces an isomorphism $K_* (A) \to K_* (B).$
Since $K_0 (A) \cong \Z \big[ \tfrac{1}{3} \big],$
Proposition~\ref{P:NoRP} shows that $\bt$ does not have the \rp.

We further check that $\bt$ does not have locally discrete K-theory.
There is an obvious isomorphism
$\ph \colon
 {\mathcal{O}}_{\infty} \otimes C^* (\Zqt, A, \af)
  \to C^* (\Zqt, B, \bt),$
which is equivariant for the dual actions,
using the trivial action on ${\mathcal{O}}_{\infty}.$
The K\"{u}nneth formula implies that
$a \mapsto \ph (1 \otimes a)$ defines an isomorphism
from $K_* ( C^* (\Zqt, A, \af))$ to $K_* ( C^* (\Zqt, B, \bt))$
which is equivariant for the dual actions.
Theorem 2.6.1 and Proposition 2.7.10 of~\cite{Ph1}
now imply that
$K_*^{\Zqt} (A) \cong K_*^{\Zqt} (B)$ as $R (\Zqt)$-modules.
Example~\ref{E:PTypeK}
shows that $\af$ does not have locally discrete K-theory.
So neither does~$\bt.$
\end{exa}

\begin{exa}\label{E:PIRPn}
Let $\af \colon \Zqt \to \Aut (A)$ be as in Example~\ref{E:RPn}.
Then $B = {\mathcal{O}}_{\infty} \otimes A$
is a unital Kirchberg algebra
which satisfies the Universal Coefficient Theorem.
The action $\af$ has the \trp\  %
(Example~\ref{E:TRP}(\ref{E:TRP:Torsion})),
so is pointwise outer by Proposition~\ref{T:TRPImpPOut}.
The proof of Proposition~4.2 of~\cite{PhtRp4}
shows that $K_0 (A)$ has a summand isomorphic to the $K_0$-group
of an odd UHF~algebra.
So $\bt$ does not have the \rp\  for the same reason as in
Example~\ref{E:PIUHF}.
However, in this example,
$\bt$ is totally K-free.
It suffices to check that $\bt$ has locally discrete K-theory.
For this, use Example~\ref{E:RPn2} and the same
argument as in Example~\ref{E:PIUHF}.
(Total K-freeness has significant content because
the K-theory of $B$ and of $C^* (\Zqt, B, \bt)$ is highly nontrivial.)
\end{exa}

\begin{exa}\label{E:O2NoRP}
Let $\bt \colon \Zqt \to \Aut ( {\mathcal{O}}_{2} )$
be as in Example~5.8 of~\cite{Iz1}.
The nontrivial group element acts
on the standard generators $s_1$ and $s_2$ of ${\mathcal{O}}_{2}$
by $s_1 \mapsto - s_1$ and $s_2 \mapsto - s_2.$
Then $\bt$ is pointwise outer by the theorem in~\cite{MT}.
In Example~5.8 of~\cite{Iz1}, it is shown that the fixed point
algebra ${\mathcal{O}}_{2}^{\bt}$ is isomorphic to~${\mathcal{O}}_{4}.$
Thus
$K_0 \big( {\mathcal{O}}_{2}^{\bt} \big) \to K_0 ( {\mathcal{O}}_{2} )$
is not injective.
So Theorem~\ref{T:RPKthInj} implies that $\bt$ does not have the \rp.
Thus,
pointwise outerness does not imply the \rp\  %
even in the absence of obstructions
of the type that appear
in Proposition~\ref{P:NoRP} and the preceding discussion.
\end{exa}

It does, however, seem reasonable to hope for a positive
solution to the following problem.
It is not clear that our definition of the \trp\  is right
for actions on unital purely infinite simple \ca s,
and some modification may be needed.

\begin{pbm}\label{Pb:OutImpTRP}
Let $A$ be a unital Kirchberg algebra,
let $G$ be a finite group
and let $\af \colon G \to \Aut (A)$ be a pointwise outer action.
Does it follow that $\af$ has the \trp?
\end{pbm}

\section{Full Connes spectrum}\label{Sec:Connes}

\indent
The noncommutative generalization of Theorem~\ref{CT:MorEq}
is saturation.
We will, however,
see that saturation is very weak,
and we will primarily consider a stronger condition,
hereditary saturation,
which can also be expressed in terms of the strong Connes spectrum.
Saturation says that the fixed point algebra is essentially the
same as the \cp.
Hereditary saturation turns out to be exactly the condition
needed to ensure that every ideal in the \cp\  is the \cp\  of
an invariant ideal in the original algebra.

The following lemma is needed as preparation for the definition
of saturation.

\begin{lem}[Proposition~7.1.3 of~\cite{Ph1}]\label{L:Bimodule}
Let $\af \colon G \to \Aut (A)$ be an action of a compact group
on a \ca~$A.$
Then the following definitions make
a suitable completion of $A$ into an $A^G$--$C^* (G, A, \af)$
bimodule,
which is almost a Morita equivalence bimodule in the
sense of the definition on page~287 of~\cite{Rf}
(the only missing condition is that the range of
$\langle \cdot, \cdot \rangle_{C^* (G, A, \af)}$ need not be dense):
\begin{enumerate}
\item\label{L:Bimodule:1}
$a \cdot x = a x$ for $a \in A^G$ and $x \in A.$
\item\label{L:Bimodule:2}
$x f = \int_G \af_g^{-1} (x f (g)) \, d \mu (g)$
for $x \in A$ and $f \in L^1 (G, A, \af).$
\item\label{L:Bimodule:3}
$\langle x, y \rangle_{A^G} = \int_G \af_g (x y^*) \, d \mu (g)$
for $x, y \in A.$
\item\label{L:Bimodule:4}
$\langle x, y \rangle_{C^* (G, A, \af)}$
is the function $g \mapsto x^* \af_g (y)$
for $x, y \in A.$
\end{enumerate}
\end{lem}

The following definition is originally due to Rieffel.
A version for proper actions
of not necessarily compact groups appears in
Definition~1.6 of~\cite{Rf3}.

\begin{dfn}[Definition~7.1.4 of~\cite{Ph1}]\label{D:Sat}
Let $\af \colon G \to \Aut (A)$ be an action of a compact group~$G$
on a \ca~$A.$
Then $\af$ is said to be {\emph{saturated}}
if the bimodule of Lemma~\ref{L:Bimodule}
is a Morita equivalence bimodule,
that is,
the range of $\langle \cdot, \cdot \rangle_{C^* (G, A, \af)}$
is dense in $C^* (G, A, \af).$
\end{dfn}

For an action of a compact group on $C_0 (X),$
it follows from Proposition~7.1.12 and Theorem~7.2.6 of~\cite{Ph1}
that saturation is equivalent to freeness of the action of $G$ on~$X.$
Saturation is a very weak form of freeness,
since even inner actions on simple \ca s can be saturated.
See Proposition~7.2.1 of~\cite{Ph1}.
For example,
this result shows that the action $\af \colon \Zqt \to \Aut (M_2)$
generated by
$\Ad
 \left( \begin{smallmatrix} 1 & 0 \\ 0 & -1 \end{smallmatrix} \right)$
is saturated.
Also see Example~\ref{E:PTypeSat}.
In particular, crossed products of simple \ca s by saturated actions
can be far from simple.
(With $\af$ as above,
$C^* (\Zqt, M_2, \af) \cong \C \oplus \C.$)
Thus, the analog of Theorem~\ref{CT:Outer} fails.
It does ensure, by definition,
that $A^G$ is strongly Morita equivalent to $C^* (G, A, \af),$
which is a noncommutative analog of Theorem~\ref{CT:MorEq}.
In particular, at least in the separable case,
it ensures that the map $K_* (A^G) \to K_* (C^* (G, A, \af) )$
is an isomorphism.
This statement is a noncommutative analog of the conclusion
of Theorem~\ref{CT:KXG}.

\begin{exa}\label{E:PTypeSat}
Let $\af \colon \Zqt \to \Aut (A)$ be as in Definition~\ref{D:PType}.
Then \tfae:
\begin{enumerate}
\item\label{E:PTypeSat:Sat}
$\af$ is saturated.
\item\label{E:PTypeSat:NTr}
For some $n,$ we have $k (n) \neq 0.$
\item\label{E:PTypeSat:NTr2}
$\af$ is nontrivial.
\end{enumerate}

Obviously saturation implies the other two conditions,
and the other two conditions are equivalent.
So assume~(\ref{E:PTypeSat:NTr}).
Choose $n_0$ such that $k (n_0) \neq 0.$
For $n \geq n_0,$
one can use Proposition~7.2.1 of~\cite{Ph1}
to see that
the action of $\Zqt$ on $\bigotimes_{k = 1}^{n} M_{d (k)}$
is saturated.
Proposition~7.1.13 of~\cite{Ph1} now implies that $\af$ is saturated.
\end{exa}

Saturation is far too weak for most purposes.
For something more useful, one must pass to hereditary saturation.

\begin{dfn}[Definition~7.2.2 of~\cite{Ph1}]\label{D:HSat}
Let $\af \colon G \to \Aut (A)$ be an action of a compact group~$G$
on a \ca~$A.$
The action is said to be {\emph{hereditarily saturated}}
if for every nonzero $G$-invariant hereditary subalgebra $B \subset A,$
the restricted action $\af_{( \cdot )} |_B$ is saturated.
\end{dfn}

For our standard product type actions,
we have the same behavior as for pointwise outerness:

\begin{exa}\label{E:PTypeHSat}
Let $\af \colon \Zqt \to \Aut (A)$ be as in Definition~\ref{D:PType}.
Then \tfae:
\begin{enumerate}
\item\label{E:PTypeHSat:HSat}
$\af$ is hereditarily saturated.
\item\label{E:PTypeHSat:Smp}
$C^* (\Zqt, A, \af)$ is simple.
\item\label{E:PTypeHSat:NTr}
For infinitely many $n,$ we have $k (n) \neq 0.$
\end{enumerate}
To see this, use
Theorem~\ref{T:SatAndStConnes}(\ref{T:SatAndStConnes:2}) below,
Corollary~\ref{T:StrConnesAndSimp} below,
and Example~\ref{E:PTypeOut}.
\end{exa}

At least for noncyclic groups,
the behavior of hereditary saturation is not as close to
that of pointwise outerness as Example~\ref{E:PTypeHSat} suggests.

\begin{exa}[Example 4.2.3 of~\cite{Ph1}]\label{E:4Gp}
Let $A = M_{2},$ let $G = (\Z / 2 \Z)^{2}$ with
generators $g_{1}$ and $g_{2},$ and set
\[
\alpha_{g_{1}}
 = \Ad \left( \begin{array}{cc} 1 & 0 \\ 0 & -1 \end{array} \right)
\andeqn
\alpha_{g_{2}}
 = \Ad \left( \begin{array}{cc} 0 & 1 \\ 1 & 0 \end{array} \right).
\]
These generate an action of $G$
such that $\alpha_{g}$ is inner for all $g \in G$, but such that
there is no homomorphism $g \mapsto u_{g} \in U (A)$ with
$\alpha_{g} = \Ad (u_{g})$ for all $g \in G.$
The crossed product $C^* (G, M_2, \af)$ is simple.
\end{exa}

Lemma~\ref{L:Inn} shows that an analog of Example~\ref{E:4Gp}
is not possible if $G$ is cyclic and $A$ is simple.

For actions on \ca s of type~I, we have the following two results.
Both are contained in Theorem~8.3.7 of~\cite{Ph1}.

\begin{thm}[\cite{Ph1}]\label{T:Type1HSatLie}
Let $\af \colon G \to \Aut (A)$
be an action of a compact Lie group~$G$
on a type~I \ca~$A.$
If $G$ acts freely on $\Prim (A),$ then $\af$ is hereditarily saturated.
\end{thm}

\begin{thm}[\cite{Ph1}]\label{T:Type1HSatCyc}
Let $\af \colon G \to \Aut (A)$
be an action of a finite cyclic group~$G$
on a type~I \ca~$A.$
Then $\af$ is hereditarily saturated
\ifo\  $G$ acts freely on $\Prim (A).$
\end{thm}

Example~\ref{E:4Gp}
shows that the converse of Theorem~\ref{T:Type1HSatLie} is false.

Hereditary saturation is closely related to the strong Connes spectrum.
Parts (\ref{D:StrongConnes:2}) and~(\ref{D:StrongConnes:3})
of the following definition are a special case
of definitions of Kishimoto.
See the beginning of Section~2 of~\cite{Ks0}.
The definitions in~\cite{Ks0} are given for the case of
an arbitrary locally compact abelian group,
and are more complicated because,
without compactness, the eigenspaces in Part~(\ref{D:StrongConnes:1})
are usually too small to be useful.
One must use approximate eigenspaces instead.
We refer to~\cite{Ks0} for the definition in that case,
but we make some comments below about what happens in the
locally compact case.

\begin{dfn}[\cite{Ks0}]\label{D:StrongConnes}
Let $\af \colon G \to \Aut (A)$
be an action of a compact abelian group~$G$
on a \ca~$A.$
\begin{enumerate}
\item\label{D:StrongConnes:1}
For $\ta \in {\widehat{G}},$
the Pontryagin dual of~$G,$
we let $A_{\ta} \subset A$ be the eigenspace
\[
A_{\ta} = \{ a \in A \colon
   {\mbox{$\af_g (a) = \ta (g) a$ for all $g \in G$}} \}.
\]
\item\label{D:StrongConnes:2}
The {\emph{strong (Arveson) spectrum}}
${\widetilde{\mathrm{Sp}}} (\af)$ of $\af$ is
\[
{\widetilde{\mathrm{Sp}}} (\af)
 = \big\{ \ta \in {\widehat{G}} \colon
     {\overline{A_{\ta}^* A A_{\ta} }} = A \big\}.
\]
\item\label{D:StrongConnes:3}
The {\emph{strong Connes spectrum}}
${\widetilde{\Gm}} (\af)$ of $\af$ is
the intersection over all nonzero $G$-invariant hereditary subalgebras
$B \subset A$
of ${\widetilde{\mathrm{Sp}}} \big( \af_{( \cdot )} |_B \big).$
\end{enumerate}
\end{dfn}

The strong Arveson spectrum is a modification of a much
older notion called the (Arveson) spectrum ${\mathrm{Sp}} (\af),$
defined for actions of compact groups
by the using the condition $A_{\ta} \neq \{ 0 \}$
instead of ${\overline{A_{\ta}^* A A_{\ta} }} = A.$
Thus, the strong Arveson spectrum is smaller.
The Connes spectrum $\Gm (\af)$ is then
as in Definition~\ref{D:StrongConnes}(\ref{D:StrongConnes:3}),
but using the Arveson spectrum instead of the strong Arveson spectrum.
The Connes spectrum was introduced by Connes
(Section~2.2 of~\cite{Cn}) for actions on von Neumann algebras.
The main early work for \ca s was done by Olesen and Pedersen.
See \cite{OPd1}, \cite{OPd2}, and~\cite{OPd3}.
Also see Sections 8.1 and 8.8--8.11 of~\cite{Pd1},
where a third version, the Borchers spectrum, is also treated.
We briefly discuss the significance of the difference after
Theorem~\ref{T:AbConnes}.
The analog of the strong Connes spectrum for von Neumann algebras
gives the same thing as the Connes spectrum
(Remark~2.4 of~\cite{Ks0}).
Some cases in which ${\widetilde{\Gm}} (\af) = \Gm (\af)$ are
discussed at the end of Section~3 of~\cite{Ks0}.

A version of the strong Arveson spectrum for actions of
compact nonabelian groups is given in Definition~1.1(b) of~\cite{GLP},
and a version of the strong Connes spectrum
is given in Definition~1.2(b) of~\cite{GLP}.
The values of both are subsets of the space ${\widehat{G}}$
of unitary equivalence classes of irreducible representations of~$G.$

The relevance here is the following theorem,
which follows from the discussion after Lemma~3.1 of~\cite{GLP}.
For abelian groups,
the first part is essentially originally due to Rieffel.
See Theorem~7.1.15 of~\cite{Ph1} and the comment after its proof,
and Theorem~7.2.7 of~\cite{Ph1}.

\begin{thm}[\cite{GLP}]\label{T:SatAndStConnes}
Let $\af \colon G \to \Aut (A)$
be an action of a compact group~$G$
on a \ca~$A.$
Then:
\begin{enumerate}
\item\label{T:SatAndStConnes:1}
$\af$ is saturated \ifo\  %
${\widetilde{\mathrm{Sp}}} (\af) = {\widehat{G}}.$
\item\label{T:SatAndStConnes:2}
$\af$ is hereditarily saturated \ifo\  %
${\widetilde{\Gm}} (\af) = {\widehat{G}}.$
\end{enumerate}
\end{thm}

Saturation,
equivalently full strong Connes spectrum,
is exactly the condition needed for every ideal in the \cp\  %
to be the \cp\  by an invariant ideal.
Combining Theorem~\ref{T:SatAndStConnes}(\ref{T:SatAndStConnes:2})
with Theorem~3.3 of~\cite{GLP},
we get:

\begin{thm}[\cite{GLP}]\label{T:NCConnes}
Let $\af \colon G \to \Aut (A)$
be an action of a compact group~$G$ on a \ca~$A.$
Then \tfae:
\begin{enumerate}
\item\label{T:NCConnes:1}
${\widetilde{\Gm}} (\af) = {\widehat{G}}.$
\item\label{T:NCConnes:2}
$\af$ is hereditarily saturated.
\item\label{T:NCConnes:3}
Every ideal $J \subset C^* (G, A, \af)$ has the form
$C^* (G, I, \af_{(\cdot)} |_I)$
for some $G$-invariant ideal $I \subset A.$
\end{enumerate}
\end{thm}

\begin{cor}\label{T:StrConnesAndSimp}
Let $\af \colon G \to \Aut (A)$ be an action of a
compact group~$G$
on a \ca~$A.$
Then $C^* (G, A, \af)$ is simple
\ifo\  $\af$ is minimal (Definition~\ref{D:Min})
and ${\widetilde{\Gm}} (\af) = {\widehat{G}},$
\ifo\  $\af$ is minimal and hereditarily saturated.
\end{cor}

See~\cite{GL} for
more on the Connes spectrum
for actions of compact nonabelian groups,
including ways in which their behavior is both like and unlike
that for actions of compact abelian groups.

One gets similar results for abelian but not necessarily compact groups.
These are due to Kishimoto~\cite{Ks0}.

\begin{thm}[Lemma~3.4 of~\cite{Ks0}]\label{T:StrConnesAndIdeals}
Let $\af \colon G \to \Aut (A)$ be an action of a
locally compact abelian group~$G$
on a \ca~$A.$
Then
\[
{\widetilde{\Gm}} (\af)
 = \big\{ \ta \in {\widehat{G}} \colon
    {\mbox{${\widehat{\af}}_{\ta} (I) \subset I$
         for all ideals $I \subset C^* (G, A, \af)$}} \big\}.
\]
\end{thm}

The following is a consequence of Theorem~\ref{T:StrConnesAndIdeals}
and Takai duality (7.9.3 of~\cite{Pd1}):

\begin{thm}[\cite{Ks0}]\label{T:AbConnes}
Let $\af \colon G \to \Aut (A)$
be an action of an abelian group~$G$ on a \ca~$A.$
Then \tfae:
\begin{enumerate}
\item\label{T:AbConnes:1}
${\widetilde{\Gm}} (\af) = {\widehat{G}}.$
\item\label{T:AbConnes:3}
Every ideal $J \subset C^* (G, A, \af)$ has the form
$C^* (G, I, \af_{(\cdot)} |_I)$
for some $G$-invariant ideal $I \subset A.$
\end{enumerate}
\end{thm}

In particular, it follows
(Theorem~3.5 of~\cite{Ks0}) that $C^* (G, A, \af)$ is simple \ifo\  %
$\af$ is minimal (Definition~\ref{D:Min})
and ${\widetilde{\Gm}} (\af) = {\widehat{G}}.$

The corresponding results using the Connes spectrum for
an action of an abelian group are that
\[
\Gm (\af)
 = \big\{ \ta \in {\widehat{G}} \colon
    {\mbox{${\widehat{\af}}_{\ta} (I) \cap I \neq \varnothing$
         for all ideals $I \subset C^* (G, A, \af)$}} \big\}
\]
(Proposition 8.11.8 of~\cite{Pd1}),
and that
$C^* (G, A, \af)$ is prime \ifo\  $A$ is $G$-prime
(any two nonzero $G$-invariant ideals have nonzero intersection)
and $\Gm (\af) = {\widehat{G}}$
(Theorem 8.11.10 of~\cite{Pd1}).
We also mention Corollary 8.9.10 of~\cite{Pd1}:
an automorphism of a simple \ca\  is inner \ifo\  the
Connes spectrum of the action of $\Z$ that it generates is $\{ 1 \}.$
We have chosen to emphasize the strong Connes spectrum because
of Theorems \ref{T:NCConnes} and~\ref{T:AbConnes}.

The \rp\  and the \trp\  imply hereditary saturation:

\begin{prp}\label{P:RPImpHS}
Let $A$ be a \uca, let $G$ be a finite group,
and let $\af \colon G \to \Aut (A)$ have the \rp.
Then $\af$ is hereditarily saturated
and ${\widetilde{\Gm}} (\af) = {\widehat{G}}.$
\end{prp}

\begin{proof}
Combine Theorem~\ref{T:RokhIdeals} and Theorem~\ref{T:NCConnes}.
\end{proof}

The result probably also holds
when $G$ is a second countable compact group,
and the statement about ${\widetilde{\Gm}} (\af)$ probably holds
when $G = \Z$ and when $G = \R.$

\begin{prp}\label{P:TRPImpHS}
Let $A$ be an infinite dimensional simple \uca,
let $G$ be a finite group,
and let $\af \colon G \to \Aut (A)$ have the \trp.
Then $\af$ is hereditarily saturated
and ${\widetilde{\Gm}} (\af) = {\widehat{G}}.$
\end{prp}

\begin{proof}
Combine Proposition~\ref{T:TRPImpPOut}, Theorem~\ref{T:Ks},
and Corollary~\ref{T:StrConnesAndSimp}.
\end{proof}

For the relationship with strong pointwise outerness,
the following easy to get results
are all we know.
However, strong pointwise outerness ought to imply hereditary
saturation in much greater generality.

\begin{prp}\label{T:HSatImpOut}
Let $A$ be a simple \uca, let $p$ be a prime,
and let $\af \colon \Zq{p} \to \Aut (A)$ be an action of
$\Zq{p}$ on~$A.$
Then $\af$ is hereditarily saturated \ifo\  $\af$ is pointwise outer.
\end{prp}

\begin{proof}
The crossed product $C^* (\Zq{p}, A, \af)$ is simple
\ifo\  $\af$ is pointwise outer,
by Lemma~\ref{L:Zp} and Theorem~\ref{T:Ks}.
Also, $C^* (\Zq{p}, A, \af)$ is simple
\ifo\  $\af$ is hereditarily saturated,
by Corollary~\ref{T:StrConnesAndSimp}.
\end{proof}

\begin{prp}\label{T:HSatImpOut2}
Let $\af \colon G \to \Aut (A)$ be a pointwise outer action of a finite
group $G$ on a simple \ca~$A.$
Then $\af$ is hereditarily saturated.
\end{prp}

\begin{proof}
Combine Theorem~\ref{T:Ks} and Corollary~\ref{T:StrConnesAndSimp}.
\end{proof}

Hereditary saturation has the following permanence properties.

\begin{prp}\label{T:HSatPerm}
Let $A$ be a \uca, let $G$ be a compact group,
and let $\af \colon G \to \Aut (A)$ be an action of $G$ on~$A.$
\begin{enumerate}
\item\label{T:HSatPerm:IffExt}
If $I \subset A$ is a $\af$-invariant ideal,
then $\af$ is hereditarily saturated \ifo\  %
$\af_{( \cdot )} |_I$
and the induced action of $G$ on $A / I$
are both hereditarily saturated.
\item\label{T:HSatPerm:DLim}
If $A = \dirlim A_n$ is a direct limit of \ca s,
and $\af \colon G \to \Aut (A)$ is an action obtained
as the direct limit of actions $\af^{(n)} \colon G \to \Aut (A_n),$
such that $\af^{(n)}$ is hereditarily saturated for all~$n,$
then $\af$ is hereditarily saturated.
\end{enumerate}
\end{prp}

\begin{proof}
Part~(\ref{T:HSatPerm:IffExt}) is Proposition~7.2.3 of~\cite{Ph1}.

Part~(\ref{T:HSatPerm:DLim}) was overlooked in~\cite{Ph1}.
For saturation, it is Proposition 7.1.13 of~\cite{Ph1}.
The rest of the proof follows the same argument as
for the proof of Proposition~\ref{T:TKFPerm}(\ref{T:TKFPerm:DLim}).
\end{proof}

One can see from Example~\ref{E:4Gp} that hereditary
saturation does not pass to subgroups,
since the nontrivial subgroups in that case
act via inner actions.
Example~\ref{E:939} shows
(see Remark 9.3.10 of~\cite{Ph1})
that saturation does not even pass to subgroups of cyclic groups.

\begin{pbm}\label{P:SatAndSubgp}
Which actions of finite groups have the property that their
restrictions to all subgroups are hereditarily saturated?
Are such actions necessarily strongly pointwise outer?
\end{pbm}

As far as we know, nobody has looked at this.
Proposition~\ref{T:HSatImpOut} might be taken as evidence
in favor of the second part.

Lemma~\ref{L:ZpTens} and Proposition~\ref{T:HSatImpOut}
imply a very special case of hereditary saturation of
the tensor product of a hereditarily saturated action
and an arbitrary action.
The general result, however, is false.

\begin{exa}\label{E:UnsatTP}
Adopt the notation of Example~\ref{E:4Gp}.
It follows from Corollary~\ref{T:StrConnesAndSimp}
that the action $\af$ is hereditarily saturated.
Let $B = \C^2,$
and define $\bt \colon G \to \Aut (B)$
by $\bt_{g_1} = \id_B$ and
$\bt_{g_2} (\ld_1, \ld_2) = (\ld_2, \ld_1)$
for $\ld_1, \ld_2 \in \C.$
Then $A \otimes B \cong M_2 \oplus M_2,$
and $\af_{g_2} \otimes \bt_{g_2}$ interchanges the summands,
so $\af \otimes \bt$ is minimal.
However, $C^* (G, \, A \otimes B, \, \af \otimes \bt)$
has vector space dimension
$\card (G) \cdot \dim (A \otimes B) = 32,$
and there is no simple \ca\  of this dimension.
Thus $C^* (G, \, A \otimes B, \, \af \otimes \bt)$ is not simple,
so $\af \otimes \bt$ is not hereditarily saturated,
by Corollary~\ref{T:StrConnesAndSimp}.
\end{exa}

\end{document}